\documentclass{article}
\usepackage[a4paper]{geometry}
\usepackage[utf8]{inputenc}
\usepackage[colorlinks]{hyperref}
\usepackage[singlelinecheck=false]{caption}
\usepackage{booktabs}

\usepackage{graphicx}
\usepackage{latexsym}

\usepackage{lipsum}
\usepackage{graphicx}
\usepackage{authblk}
\usepackage[english]{babel}
\usepackage{color}
\usepackage{hyperref}
\hypersetup{
    colorlinks=true,
    linkcolor=blue,
    filecolor=blue,      
    urlcolor=blue,
    citecolor=blue,
    }
\usepackage{dsfont}
\usepackage{geometry}
\usepackage{float}
\usepackage{bbm}
\usepackage{enumerate}

\usepackage[nottoc]{tocbibind}
 \usepackage{indentfirst}

\graphicspath{{./figure/}}

\usepackage{pgfplots}
\usepackage{comment}
\usepackage{fontenc}
\usepackage{booktabs}
\usepackage{tabularx}
\usepackage{tikz}
\usepackage{caption}
\usepackage{subcaption}

\DeclareCaptionLabelFormat{custom}
{%
      \textbf{Figure #2}
}
\DeclareCaptionLabelSeparator{custom}{ }
\DeclareCaptionFormat{custom}
{%
\centering
    #1 #2 #3 
}
 \usepackage{tikz-cd}

\usepackage{amsmath,amsfonts,amssymb,amsthm}

\numberwithin{equation}{section}
\usepackage{etoolbox}
\apptocmd{\thebibliography}{}{}{}
\renewcommand*{\Re}{\mathop{}\!\mathrm{Re}\,} %for real part

\theoremstyle{plain}
\newtheorem{theorem}{Theorem}[section]
\newtheorem{lemma}[theorem]{Lemma}
\newtheorem{corollary}[theorem]{Corollary}
\newtheorem{proposition}[theorem]{Proposition}

\newtheorem{reminder*}{[theorem]Reminder}

\newtheorem{details*}[theorem]{Details}

\newtheorem{comm*}{Comment}
\newtheorem{definition}[theorem]{Definition} 
\newtheorem{definition*}{[theorem]Definition}

\newtheorem{notation*}{Notation}

\title{A stochastic version of the Hopfield-Ninio kinetic proofreading model}
  \author{E. Franco, J. J. L. Velázquez}

\begin{document}

\maketitle

\tableofcontents
\begin{abstract}
     In this paper we study a simple stochastic version of the Hopfield-Ninio kinetic proofreading model. 
     The model is characterized by means of two parameters, the unbinding time, which depends on the binding energy between a ligand and a receptor, and the number of times $M \geq 1$ that a ligand attaches to a receptor. 
We prove that, under suitable assumptions on $M$, our model has an extreme specificity, i.e. it is capable to discriminate between different ligands, and a high sensitivity, i.e. the response of the system does not change in a significant manner for ranges of ligands varying within several orders of magnitude.
Additional quantities like the amount of energy used by the network or the time required to yield a response will be also computed.
We also show that our results are robust, i.e.~they do not depend on the specific choice of parameters that we make in this paper. 
\end{abstract}

\section{Introduction}

\subsection{Kinetic proofreading models}
The goal of this paper is to study the properties of a simple stochastic version of the Hopfield/Ninio kinetic proofreading (KPR) model. 
The classical kinetic proofreading chemical networks have been proposed by Hopfield and Ninio in the 1970's, (see \cite{Hopfield, Ninio}) in order to find a possible explanation of the very low error rate ($10^{-4}$) in protein synthesis and in DNA replication (error rate $10^{-9}$). 
We refer to \cite{alon2019introduction} for a review on the topic.

The main feature of the Hopfield/Ninio chemical networks is that they allow to discriminate between ligands, that have different affinities with a receptor, in a much more efficient manner than the one that could be expected in a system in equilibrium at constant temperature.
This much larger discrimination efficiency can be achieved because the networks operate out of equilibrium, and there is a continuous influx of ATP in the system, as well as a continuous outflux of ADP, that feeds a sequence of phosphorylation reactions.

The kinetic proofreading mechanisms introduced by Hopfield and Ninio were adapted by McKeithan (see \cite{mckeithan1995kinetic}) in order to explain the ability of $T$-cell receptors to discriminate between different antigens. 
Indeed, the acquired (or adaptive) immunity relies on the capability of immune cells to distinguish between the body's own antigens and foreign antigens.
When a ligand attaches to a $T$-receptor it may or may not trigger a response of the immune system. We assume that a self-antigen, which is an antigen (usually a protein) that is produced by the organism to which the $T$-cell belongs, typically does not trigger an immune response. While foreign antigens, that are produced outside the organism, usually trigger the immune response.  

The kinetic proofreading mechanism is characterized by a reversible reaction yielding the attachment between the ligand and the receptor, followed by a sequence of phosphorylation events that take place in a non-reversible manner, spending an ATP molecule in each phosphorylation step.
In the context of immune recognition the receptor is typically a receptor of a T-cell located on the surface of the $T$-cell, while the ligand is a pMHC, i.e.~a complex made of a major histocompatibility complex (MHC) and a peptide.
We refer to \cite{janeway2001immunobiology} for a detailed description of the functioning of  MHCs. 
The T-cells scan the set of peptides presented by the MHCs and must trigger the reaction of the immune system only if one of the peptides is recognised as specific of a foreign antigen. The kinetic proofreading mechanisms is therefore needed to avoid errors in the recognition process.

The Hopfield/Ninio kinetic proofreading network is usually modelled using systems of ODEs. It is then possible to compute a stationary solution under the assumption that the concentration of ligands is constant. 
These stationary solutions can be interpreted as the solutions yielding the flux of the complexes receptor-ligand to the maximally phosphorylation  state. 
The kinetic proofreading network can be thought as a switch that, upon the attachment of a ligand to the receptor, yields a response after a waiting time. 
These stationary solutions allow to compute also the probabilities at which the responses take place for a given ligand. 

These probabilities of response are functions of the binding energies (or characteristic unbinding times) between the ligand and the receptor. 
Suppose that two ligands have binding energies $E_1 $ and $E_2 $ respectively. Let us denote with $p_1 $ the probability of response of the first ligand and with $p_2 $ the probability of response of the second ligand.
It turns out, see for instance \cite{Hopfield}, that, if the system works in equilibrium, then 
\begin{equation} \label{bad ration}
\frac{p_1}{p_2} = e^{(E_1-E_2)/k_B T } 
\end{equation}
where $k_B $ is the Boltzmann constant and $T$ is the temperature in the system. In all this paper we will always use units of energy so that $k_B T =1 $. 
Notice that the ratio $\frac{p_1}{p_2} $ is sometimes called the potency of the system, see \cite{altan2005modeling}.

It was already noticed (see \cite{Hopfield} and references therein) that this would not be enough to explain the low error rate in processes like protein synthesis and in DNA replication. 
Similarly, it is shown in \cite{mckeithan1995kinetic} that this differences of energies $E_1, E_2 $ are not large enough to allow the immune system to discriminate between foreign antigens and self-antigens. 
If we consider, instead, a kinetic proofreading with $N$ steps (phosphorylation events) we improve the discrimination capability, obtaining that 
\begin{equation} \label{compare fluxes}
\frac{p_1}{p_2} = \left(e^{E_1-E_2 } \right)^N , 
\end{equation}
instead of \eqref{bad ration} (we recall that we are using units of energy of $k_B T $). 
The improvement in \eqref{compare fluxes} was noticed in \cite{mckeithan1995kinetic} and, for $N=2$, already in \cite{Hopfield, Ninio}.

Many deterministic generalizations of the Hopfield/Ninio kinetic proofreading chemical network have been studied, including models with feedback and with an inhibition effect (as \cite{chan2004feedback,franccois2013phenotypic,rendall2017multiple}), with rebinding (see \cite{dushek2009role,dushek2014induced}) and time dependent versions of the Hopfield model (see for instance \cite{franco2023description,sontag2001structure}). All these deterministic approaches, based on systems of ODEs, are valid in order to describe systems of many receptors interacting with many ligands. 
The solutions of the ODEs then yield the dynamics of the concentrations of the complex ligand-receptor in different phosphorylation  states.

However, there is experimental evidence that T-cells have a very high degree of sensitivity and low ligands (pMHCs) densities corresponding to foreign antigens can trigger the response, see for instance \cite{altan2005modeling,huang2013single}. 
Therefore, in order to describe the individual response of one receptor to a ligand, a stochastic approach is more natural. This is the reason why we study a probabilistic model of kinetic proofreading. 
We refer to \cite{chan2004feedback,currie2012stochastic,kirby2023proofreading} for some other probabilistic models of kinetic proofreading.
In particular, the model studied in this paper has analogies with the models studied in \cite{bel2009simplicity,currie2012stochastic,munsky2009specificity}. However, here we focus in the study of the discrimination properties of the model in suitable asymptotic regimes. 
We stress that other stochastic models have been extensively used in order to model other biological systems, for instance to model enzyme dynamics for single molecule reactions, models of single ion-channels and also in models of cell polarization  (see \cite{altschuler2008spontaneous,english2006ever,qian2002single,xie2001single}).

\begin{figure}[H] 
\centering
\includegraphics[width=0.5\linewidth]{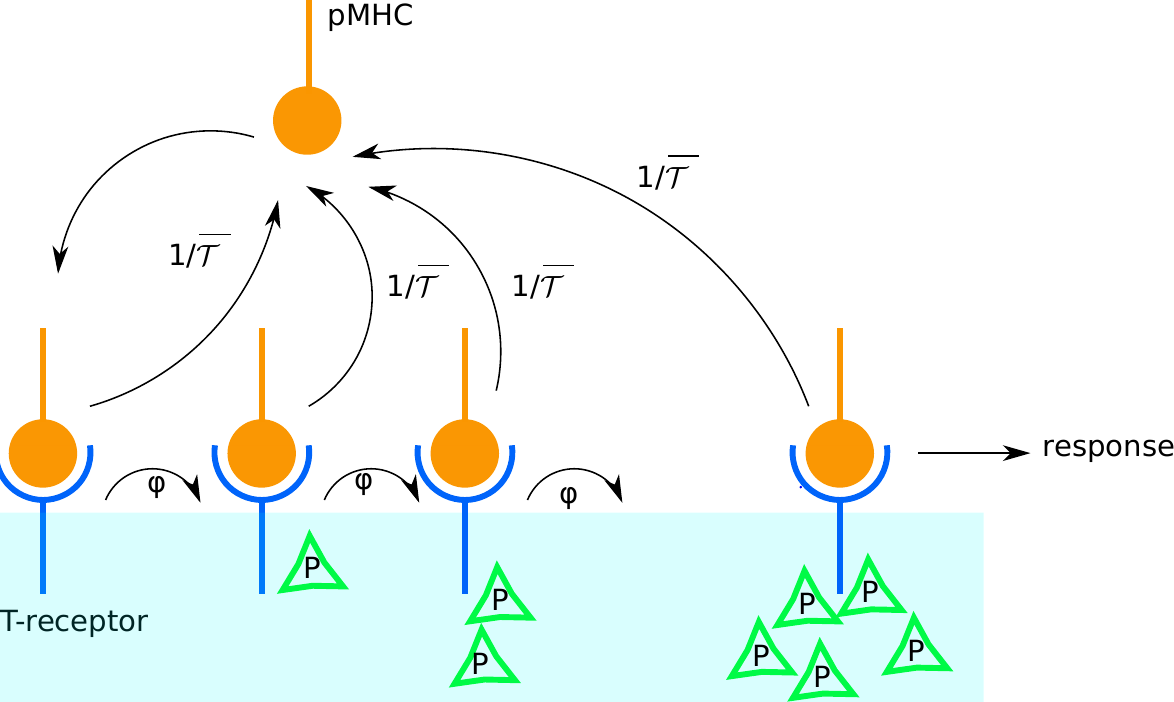}
\caption{KPR model studied in this paper
}\label{fig1}
\end{figure}

\subsection{A probabilistic kinetic proofreading model}

Several changes to the classical KPR network formulated by McKeithan have been suggested in order to improve its capability to discriminate between different ligands (see for instance \cite{franccois2016case,kirby2023proofreading,lever2014phenotypic,xiao2024leisure}).
For example, a model of KPR including an inhibitory loop in the last KPR step has been proposed in \cite{franccois2016case}, while a model with rebinding has been formulated in \cite{altan2005modeling}. 
In this paper, we analyse a simple probabilistic model of kinetic proofreading, in which a ligand is allowed to bind to the receptor $M $ times. As we will see, this simple model has (under specific assumptions on the number $M \geq 1$ of times that a ligand can reattach to the receptor) the key properties that are expected from a model of $T$-cell activation.  

We study a stochastic kinetic proofreading model using the theory of Markov processes.
This allows us to describe the evolution of individual complexes ligand-receptors that can be at different phosphorylation  states. 
We will introduce a Markov process yielding the probability of having a complex receptor-ligand at a certain phosphorylation  state at  time $t>0$, given that the ligand attached to the receptor at time $t=0$. 

We will denote the transition rate from one phosphorylation  state to another one as $\varphi$. We will assume that $\varphi$ does not depend on the phosphorylation state of the complex receptor-ligand. Moreover, we assume that the phosphorylation rate $\varphi$ does not depend on the type of ligand. 
On the other hand, the probability of unbinding (or detachment rate) of the ligand from the receptor is $1/\overline \tau $.

One of the properties characterizing the Markov process will be then the value of $\tau=\varphi \overline \tau $, the unbinding time in units of the phosphorylation time $1/\varphi$. 
However, in order to understand the robustness of the model under perturbations of the parameters, we analyse, in Section \ref{sec:different rates}, a model in which the phosphorylation rates and the unbinding times are dependent on the phosphorylation state of the complex ligand-receptor. 

The reaction describing the attachment of the ligand to the receptor can be formulated as 
\begin{equation} \label{attch-detach}
L + R  \overset{k_a} { \underset{k_d} \rightleftarrows}  LR 
\end{equation}
where $L$ is the ligand (i.e.~the complex pMHC) and $R$ the $T$-cell receptor. 
We denote with $LR $ the complex ligand-receptor, that will be denoted later as the state $(1)$. 
We denote the rate of attachment and detachment as $k_a $ and $k_d $ respectively. 
We remark that, using suitable time units, we can assume that $k_a = 1 $. 
On the other hand, the detailed balance condition allows to write 
\[ 
k_a = k_d e^{ E }
\]
where $E $ is the binding energy between the ligand and the receptor (measured in units of $k_B T $). 
This binding energy is different for different ligands and it is the main property characterizing the affinity of the ligand with the receptor.
An equivalent way of characterizing the tendency of the ligand to attach to the $T$-cell receptor is the average unbinding time, that we will denote as $\tau$. 
Then, since we are assuming that $k_a = 1 $ we have that 
\begin{equation} \label{time and energy}
\tau = e^E 
\end{equation} 
Notice that larger $\tau $ correspond to larger binding energies $E$. 

In this paper we will not include explicitly in the model the attachment reaction $L + R \rightarrow LR$ (with rate $k_a$). Instead, we will assume, as indicated above, that for a given ligand, the complex $LR $ is formed $M $ times (perhaps with different receptors) and we will examine how the dynamics of the system depends on the number $M $ of times that a ligand attaches to the receptor.
On the other hand, we assume that the detachment rate for the phosphorylated states is basically the same for all the phosphorylation states. 
This is justified by the separation, through the membrane, of the receptor and the phosphorylation sites. 
Since the detachment of the ligand from the receptor is independent on the phosphorylation state, then the detailed balance condition of the reactions detachment/attachment implies that the probability of attachment of a ligand to a phosphorylated receptor is very small and it will be neglected in this paper.

A second parameter characterizing our model is $M$, the number of times that a ligand, i.e.~a pMHC, attaches to a dephosphorylated receptor in order to attempt to produce a response.  
This number might depend on many factors, for instance on the lifetime of the pMHC and the motion of the T-cells, among others. 
Unfortunately, we are not aware of experimental measurements of $M$.
In this paper we assume that the number of trials $M$ depends on $N $, the number of steps (phosphorylation events) that take place during the kinetic proofreading process. 

We will assume that $M$ is deterministic, however $M$ is, in general, a random variable. It would be possible to study a model in which $M$ is a random variable described by a probabilistic law.
In this case the different quantities computed in the paper, as for instance the critical unbinding time and the number of ATP molecules consumed by the KPR, would be random variables.
It would be possible to compute statistical properties of these quantities averaging with respect to the probability measure characterizing $M$.
We can then compute different quantities taking averages with respect to the probability measure characterizing $M.$

We are interested, in particular, on the behaviour of the system as $N \to \infty $.
It turns out that interesting asymptotic behaviours arise assuming that $M $ is a function of $N$. 
We will analyse different possible properties of the network  for different choices of the function $M = M(N)$.   

We stress that there could be alternative ways of introducing the reattachment of the ligands to the receptor in the model.
For instance formulating a closed model with reattachment of the ligand to the receptor, i.e.~including in the model the attachment reaction in \eqref{attch-detach}. 
However, in this paper, we will focus on analysing how the properties of the kinetic proofreading change depending on $M$.

\subsection{Overview of the main results obtained in this paper}

We derive precise estimates for the probability of response of the stochastic kinetic proofreading model in terms of the unbinding time $\tau $, the number of trials $M$ and the number of ligands $L$, when the number of kinetic proofreading steps $N$ is large.
 This will allow us to analyse some features of our model. 
Since there is experimental evidence that T-cells respond quickly to a signal with very high specificity and very high sensitivity, a mathematical model of T-cells activation must necessarily satisfy these three properties (specificity, speed, sensitivity), that are referred often as \textit{golden triangle} according to the terminology introduced by \cite{feinerman2008quantitative} and used also by \cite{franccois2016case}. 
 We will examine in detail whether our model has these features if the number $N $ of KPR steps is large. 
Moreover, we also compute the amount of energy consumed by the KPR process as well as the fluctuations of the flux of product of the network, when the number of trials $M $ is large. 

Finally we will also compare our KPR model with other models, in particular with a simple model in which the KPR chain of phosphorylation events is substituted by a simple delay reaction. We will show that the model with delay is not capable of discriminate between different ligands efficiently. 
We will also study briefly a generalization of our KPR in which the detachment time and the phosphorylation rates depend on the phosphorylation state of the complex ligand receptor. 

In this paper we will use the following notation. 
We use the notation $ f \sim g $ as $N \to \infty $ to indicate that $f $ and $g $ are asymptotically equivalent, i.e.~$\lim_{N \to \infty} \frac{f }{g } =1 $. 
We use the notation $f \approx g $ to say that there exists a constant $C>0 $ such that $ \frac{1}{C } \leq \frac{f }{ g } \leq C  $. 
Similarly, we use the notation $f \gtrsim g $ to indicate that there exists a constant $C>0 $ such that $ C f \geq g  $.
Finally we say that $f \ll g $ as $N \to \infty $ if $\lim_{N \to \infty } \frac{f }{g }=0.$
\bigskip

\textbf{Specificity of the model: sharp transition from non-response to response}

\bigskip

We say that a model satisfies the \textit{specificity} property if there exists a \textit{critical unbinding time} $\tau_c $, depending on $M$ and on $N$, such that, as $N \to \infty$, the probability that the attachment of a ligand to a receptor triggers a response changes abruptly from zero (non-response) to one (response) when $\tau $ is close to $\tau_c $. 
Therefore a system satisfies the specificity property if it is able to distinguish, with an extremely low error rate, between a self-antigen and a foreign antigen. 

 In this paper, we will study the specificity of our model in Section \ref{sec:specificity}.  We analyse the probability that a ligand yields a response in terms of the unbinding time $\tau $ and of the number of trials $M$, when the number of kinetic proofreading steps $N$ is large and when the number of ligands is equal to $1$. 
We prove that, for every possible behaviour of $M $ as a function of $N$, there exists a critical unbinding time $\tau_c $, separating the region of unbinding times leading to response ($\tau > \tau_c$) from the region of unbinding times that do not lead to response ($\tau_c > \tau $).
The critical unbinding time $\tau_c$ is such that $\tau_c \rightarrow \infty $ as $N \to \infty $ when $ \log(M) \ll N $ as $N \to \infty $, while it is of order $1 $ when $ \log(M) \approx N $ as $N \to \infty $ and $\tau_c \rightarrow 0 $ as $N \to \infty $ when $ \log(M) \gg N $ as $N \to \infty $. 

Moreover, we prove that when $M \to \infty $ as $N \to \infty $ we have that, up to a suitable scaling of the time variable (i.e.~for $\tau - \tau_c = x \varepsilon_{N,M}$ where $\varepsilon_{N, M } $ could be a constant or could converge to zero as $N \to \infty$ depending on the asymptotics of the critical unbinding time $\tau_c $), the probability of response is given by $1-e^{-e^{x}}$ for $x \in \mathbb R$ as $N \to \infty $. 
Hence we have a sharp transition from response to non-response when $M \to \infty $ as $N \to \infty $. 
See also Figure \ref{fig3}. 
Instead when $M \approx 1 $ we obtain that, up to a suitable scaling, the probability of response is $1-\left(1- e^{-\frac{x}{\log(M)}}\right)^M$ as $ N \to \infty $. Also in this case we have a transition from non-response to response, but the transition is less sharp than when $M \to \infty $ as $N \to \infty $, in the sense that the range of values of $\tau$ in which the transition from non-response to response takes place is wider.  

 For all the possible behaviours of $M $ the transition from non-response to response takes place in a region of thickness equal to $ \frac{1+\tau_c }{N} $, i.e.~the model satisfies the specificity property when the unbinding time $\tau $ is such that 
 \begin{equation} \label{uncertanty principle time}
\left|  \frac{\tau}{\tau_c} - 1 \right| \gtrsim \frac{1+\tau_c }{N }
 \end{equation}
 as $N \rightarrow \infty $. 
When \eqref{uncertanty principle time} is not satisfied the kinetic proofreading model that we consider does not have the capability to discriminate between ligands leading to a response and ligands that do not lead to a response.
Notice that \eqref{uncertanty principle time} can also be written in terms of energies: the difference between the binding energy of the ligand with the receptor, $E$, and the critical binding energy $E_c$ (corresponding to the critical unbinding time $\tau_c = e^{E_c}$ as in \eqref{time and energy}) allows for discrimination if
\begin{equation} \label{uncertanty principle}
|E- E_c | \gtrsim \frac{1+\tau_c } {N }.
\end{equation}
Hence $\frac{1+\tau_c } {N }$ can be interpreted as the minimal discrimination energy that allows the network to satisfy the specificity property. 
This minimal energy depends on the asymptotic behavior of $\tau_c $ as $N \rightarrow \infty $, hence of $M$.

We stress that the thickness of the transition region from non-response to response, depends on the value of $M$ and in particular on its dependence of $N$.
When $M$ is of order one as $N \to \infty $ the thickness of the region of transition is of order one because $\tau_c \sim N $ as $N \to \infty$. 
Instead, when $M \gg 1 $ and $\log(M) \ll N  $, the thickness of the region of transition is of the order of $\frac{1}{\log(M)}$ because in this case $\tau_c \sim \frac{ N }{ \log(M) }$ as $N \to \infty$. 
When $ \log(M) \approx N $ as $N \to \infty $ then, since $\tau_c \approx 1 $, the thickness of the region where we have lack of specificity is $1/N$, hence tends to zero as $N \to \infty$. 
Finally, when $ \log(M) \gg N $ as $N \to \infty $ then, since $\tau_c \rightarrow 0 $ as $N \to \infty $, the thickness of the region where the transition from non-response to response takes place is $1/N$, hence tends to zero as $N \to \infty$. 
The specificity of the model improves as the number of trials increases, as it could be expected.

We also compare our kinetic proofreading model with a simple model with a delay. when the number of ligands is equal to $1$.  
In this model with a delay the chain of the $N $ phosphorylation reactions is substituted with one single reaction taking place at rate $1/N$, i.e.~yielding to a delay of time equal to $N$. 
We deduce that the KPR model exhibits a strong exponential dependence, for some values of $\tau $, that does not appear for models with delay. This exponential dependence is the root of the specificity property of the KPR models (see Section \ref{sec:delay}).

Finally, we consider a generalization of our KPR model in which the phosphorylation rate and the unbinding rate depend on the state of the complex ligand-receptor. 
We prove that this system also satisfies the specificity property, hence this property is robust and does not depend on the specific choice of the rates that we make. 

\bigskip 
 
\textbf{Sensitivity of the model: independence of the critical unbinding time on the number of ligands}

\bigskip

 In this paper we say that a kinetic proofreading model satisfies the \textit{sensitivity} property if the critical unbinding time $\tau_c$ is independent on the number of ligands. 
 This means that the probability that a set of ligands triggers a response does not depend in a significant manner on the number of ligands. 
The detailed asymptotic analysis of the probability of response, in Section \ref{sec:specificity}, is done analysing first the response of the model to one ligand, i.e.~we assume that the number of ligands $L$ is equal to $1$. 
We then consider the case of $L> 1 $ and analyse the changes in the critical unbinding time and in the asymptotics of the probability of response. 
We prove that, if the number of ligands is such that $\log(L) \ll N $ and $ \log(M) \approx N $ as $N \to \infty$, then the critical unbinding time is basically independent on the number of ligands.
Hence in this case the kinetic proofreading model satisfies the sensitivity property. 
We analyse also the case in which $\log(M) \ll N  $ as $N \to \infty $ and $\log(M) \gg N  $ as $N \to \infty $. In particular, when $\log(M) \ll N  $ and $ \log (L) \ll N $ as $N \to \infty $, then we obtain that $\tau_c \rightarrow \infty $ as $N \to \infty $. This happens also when the number of ligands $L $ is equal to $1$. Similarly, when $\log(M) \gg N  $ and $ \log (L) \ll N $ as $N \to \infty $ then $\tau_c \rightarrow 0$ as $N \to \infty $, as we obtained for $L=1$ (see Section \ref{sec:sensitivity} for precise results). 

\bigskip 

\textbf{Speed of KPR}

\bigskip 
Finally, the \textit{speed} of the kinetic proofreading process is the time, $T_{response}$,  that a ligand takes to produce a signal. 
Notice that $T_{response}$ is given by 
\[
T_{response} =T_{receptor} + M  T_{reattachent}
\]
where $T_{reattachent}$ is the time that a ligand takes to attach to a receptor after having detached. Therefore, $T_{receptor}= T_{response} - M  T_{reattachent} $ is the time that a ligand spends in the network, i.e.~the time that a ligand spends attached to the receptor, in order to produce a signal leading either to response or to non-response. 
Notice that $T_{receptor}$ is a random variable. 
In Section \ref{sec:time}, we compute the probability distribution, conditional to having a response, of $T_{receptor}$ as $N \to \infty $ and when $M \approx e^{b N } $ as $ N \to \infty $. 
It turns out that, if $ \left( 1- \frac{\tau}{\tau_c }  \right)  \gg \frac{1}{N} $ as $N \to \infty$ (i.e.~when we do not expect a response), then 
\begin{equation} \label{Tuniform}
T_{receptor} \sim \mathcal U (0, \tau M) \text{ as } N \to \infty 
\end{equation}
where with $\mathcal U $ we denote the uniform probability distribution on the interval $[0, \tau M ]$. 
If instead, $   \frac{\tau}{\tau_c } -1 = \frac{\xi }{N}  $ for $\xi >0 $ (i.e.~when we expect a response), then 
\begin{equation} \label{Tdelta}
T_{receptor} \sim \text{Exp} \left( \frac{e^\xi }{\tau M }\right)  \text{ as } N \to \infty, 
\end{equation} 
where $Exp(\lambda) $ denotes the probability distribution of an exponentially distributed random variable with density $t \mapsto \lambda e^{- \lambda t } $ for $t >0$. 

Therefore for large values of $N $ (and for $M \approx e^{b N } $)  we have that 
\[ \mathbb E[T_{receptor}] = \tau M  e^{- \xi }
\] where $\xi = N \left( \frac{\tau}{\tau_c}-1 \right) $. Instead when 
 $(\frac{\tau}{\tau_c} -1 ) \gg  \frac{1}{N} $ as $N\to \infty $, we have that 
 \[ \mathbb E[T_{receptor}] = \frac{\tau M } {2} .
 \] 
Therefore, on average, ligands with larger unbinding times with the receptor are faster in yielding a response of the immune system.

\bigskip 

\textbf{Energy consumption}

\bigskip 
 
A relevant quantity that we compute (see Section \ref{sec:energy}) is the total energy consumed by the kinetic proofreading mechanisms. 
In particular we compute the number of molecules of ATP that are consumed during $M$ trials of a KPR network of $N$ steps. 
The number of molecules of ATP that are consumed is a random variable $T_N $ which is described by different probability distributions for different behaviours of $M$ as $N \to \infty $. 
In particular, when $M=1 $ it turns out that the probability of spending $ k $ ATP molecules is described by a mixed random variable, i.e.~a random variable with a probability distribution that has a continuous and a discrete part. We refer to \eqref{eq:mixed density M=1} for the expression of the probability distribution. 
We find a similar situation when $M \approx 1 $ as $N \to \infty $ and $M > 1 $ (see \eqref{eq:mixed density M >1}).
Instead, when $M \to \infty $ as $N \to \infty $ we obtain (see \eqref{normal approx energy}) that 
\begin{equation} \label{Atp consumed}
 \mathcal E_N \sim \mathcal N \left( \tau M , \sqrt{\tau (1+\tau) M}  \right) \text{ as } N \to \infty
\end{equation}
where we denote with $\mathcal N (m , \sigma ) $ the probability distribution function of a normally distributed random variable with mean $m $ and variance $\sigma^2$. 

Our analysis of the asymptotic behaviour of the probability of producing $k$ molecules of ATP, allows to compute the average number of molecules of ATP that are consumed during a kinetic proofreading process for different values of the number of trials $M$. 
Since the energy consumption is a quantity that can be measured experimentally, comparing the average quantity of energy consumed, computed experimentally, with the different possible average values that we obtain from our analysis can give an estimate of the value of $M $.

\bigskip 

\textbf{Fluctuations of the fluxes}

\bigskip 

Since we are assuming that a ligand can bind to the receptor $M \geq 1 $ times, it is also possible to compute the probability $q_k$ of having $ 1 \leq k \leq M$ responses after $M $ trials of $N$ steps. Since we are interested in computing the number of responses that a ligand can produce in $M $ trials it is relevant to study the asymptotics for the probability densities $q_k$ as $M \rightarrow \infty $. 
This allows us to compute the total flux of product $J(\tau) = k/M$ that a reservoir of ligands yield. Notice that, since the number of responses $k$ is a random variable, also the flux $J(\tau) $ is a random variable. 

It has been argued in \cite{kirby2023proofreading} that the fluctuations of these fluxes do not allow to discriminate between different values of $\tau $.  On the contrary, it is claimed in \cite{xiao2024leisure} that, when we assume that the time needed for the KPR network to produce a response is large enough, the fluctuations of the fluxes become negligible. 
In this paper, we formulate precise conditions on the unbinding time $\tau $ and on the number of trials $M $ that allow to discriminate between ligands with different unbinding times using the information about the fluxes of product induced by the ligands.

Moreover, we infer that,
when  $p(\tau_1) M \to \infty $ and $ p(\tau_2) M  \to \infty $ as $M \to \infty $, and when the unbinding times $\tau_2 > \tau_1 $ are such that $\tau_1 ,  \tau_2 \approx 1 $ and satisfy 
 \begin{equation} \label{overlap}
\frac{\tau_2 - \tau_1 }{\tau_2 } \gtrsim \frac{1+\tau_1 }{N } \frac{1}{\sqrt{M p (  \tau_2) }}
\end{equation}
as $M \to \infty $ and $N \to \infty $, then the KPR system satisfies the specificity property, i.e., if $\tau_1 - \tau_2 $ satisfies \eqref{overlap}, then it is possible to determine $\tau_1 $ and $\tau_2 $ from the values of the fluxes $J(\tau_1) $ and $J(\tau_2) $. 
Therefore, \eqref{overlap} can be seen as the analogous of \eqref{uncertanty principle} for the fluxes of products.

\section{Probabilistic kinetic proofreading model}
\label{sec:model}

In this section, we describe the stochastic model of kinetic proofreading that we study in this paper. 
The main feature of this model is that we assume that, if a ligand detaches from the complex ligand-receptor, then it can attach another time to the receptor. We assume that this reattachment can take place $M \geq 1 $ for each ligand.

The main quantity that we want to study is the probability that a set of $L \geq 1 $ ligands yield a response.
We assume that the iterations of the kinetic proofreading chain are mutually independent. Therefore if $p (\tau ) $ is the probability that one ligand yields a response after one kinetic proofreading trial, then the probability that one ligand triggers the response of the immune system after $M $ trials is just 
\begin{equation}\label{p(tau)_M}
p_{M}(\tau)= 1- (1- p (\tau))^M. 
\end{equation}
Similarly, since we assume that the kinetic proofreading trials of the ligands are mutually independent random variables, if we have a set of $L \geq 1$ ligands, where each of the ligands can reattach $M \geq 1 $ times to the receptor, then the probability of this set of ligands to trigger a response is given by 
\[
p_{ML} (\tau)=1- (1-p(\tau))^{ML}. 
\]
Therefore in order to analyse our model, we need to compute $p(\tau) $, the probability  that a ligand yields a response in one kinetic proofreading trial. 
Elementary probabilistic arguments imply that 
\begin{equation} \label{p(tau)}
p(\tau)= \left( \frac{\tau }{ 1+\tau } \right)^N.
\end{equation}
To obtain \eqref{p(tau)} we explain in detail the reactions taking place during a kinetic proofreading trial. 

\subsection{Probability of response after one kinetic proofreading trial of $N $ steps}
We assume first that the number of ligands is equal to one ($L=1$) and we describe the chain of reactions that take place during one kinetic proofreading trial. 
When a ligand binds to the receptor, the complex ligand-receptor undergoes a sequence of kinetic proofreading steps. 
We assume that the set of the possible states at which this complex can be is $\{ 1, \dots,  N\} $. In other words the number of kinetic proofreading steps is $N$.  

We assume that during the kinetic proofreading process only two types of reactions take place. 
One of the reactions is phosphorylation, i.e.~the reaction 
\[ 
(k) \rightarrow (k+1) 
\] 
for $k=1, \dots, N-1 $. 
We assume that phosphorylation takes place at a rate $\varphi$, that does not depend on the state $k$. 

The other reaction that can take place is the detachment 
of the ligand from the complex ligand-receptor at each phosphorylation state $k $ for $k = 1, \dots, N $, i.e. 
\[ 
(k) \rightarrow \emptyset. 
\]
This reaction takes place at rate $1/\overline \tau$. 
In particular, $\overline \tau  $ is independent on the number of phosphates linked to the complex, i.e.~on the phosphorylation state $k$ of the complex ligand-receptor. 
The unbinding time $\overline \tau$ is different for different types of ligands. 
As indicated in the introduction, we assume that the reverse reaction, i.e.~the attachment of a ligand to a phosphorylated state with $k \geq 2 $ cannot take place.

Finally we assume that the ligand that is released during the reaction $(k) \rightarrow \emptyset$ can reattach to a receptor and, hence, can start another kinetic proofreading process. 
Each ligand can attach to the receptor $M $ times, where the number of trials $M$ is assumed to depend on the number of phosphorylation events that take place during kinetic proofreading. 
Finally, we assume that, when the system reaches the state $N$ it produces the product/response $R$ at rate $\varphi$.

\begin{center}
\begin{tikzcd}
&   & &  \text{ unbound receptor}    &   &   \\
& C_1 \arrow[urr,  "1/ \overline{\tau}"] \arrow[r, swap, "\varphi"] &  C_2  \arrow[ur,  swap, "1/ \overline{\tau}"] \arrow[r, swap, "\varphi"]  & \ldots  \arrow[r, swap, "\varphi"]  & C_{N-1} \arrow[ul,  "1/ \overline{\tau}"]  \arrow[r, swap, "\varphi"]  & C_N \arrow[ull, swap , "1/ \overline{\tau}" ] \arrow[r, swap, "\varphi"] &  \text{response}
\end{tikzcd}
\end{center}

Let $n_k(\overline t)$ be the probability that a ligand reaches the phosphorylation state $k$ in the time interval $(0,\overline t)$ and  $R(\overline t)$ be the probability of a response in the time interval $(0, \overline t)$. 
Then,
\begin{align} \label{ODE kpr overline t}
   \frac{d n_1 }{ d \overline t } =&  - \left( \varphi + \frac{1}{ \overline \tau} \right) n_1, \nonumber \\
    \frac{d n_k }{ d \overline t } =& \varphi n_{k-1}  - \left( \varphi + \frac{1}{ \overline \tau} \right) n_k, \quad k = 2 \dots N \\
    \frac{d R }{ d \overline t } =& \varphi n_N  \nonumber 
\end{align}
with initial condition $n(0) = e_1 \in \mathbb R^{N}$, where $e_1=(1, 0, \dots, 0)$ and $R(0)=0$.

In other words, a complex at state $k$ can either disappear from the system due to the detachment of the ligand from the receptor, or it disappears due to a phosphorylation  event. Hence we have a "death" rate equal to $ \varphi + \frac{1}{ \overline \tau }$ for each complex with state $k \in \{ 1, \dots,  N \} $. Finally, when a complex at state $k$ undergoes a phosphorylation reaction a complex with state $k+1 $ is formed. Hence we have a "birth" rate equal to $\varphi $ of complexes at state $k \in \{ 2, \dots  , N\} $. Finally, the complex at state $N $ produces a response at rate $\varphi $. The choice of the initial condition $n(0) = e_1 $ accounts for the fact that we are assuming that the ligands attach only to a dephosphorylated receptor.

Changing the units of time from $\overline t$ to $ t= \varphi \overline t $, i.e.~taking the phosphorylation time as unit of time, and rescaling also the detachment time $\tau = \overline \tau \varphi $ we obtain that
\begin{align} \label{ODE kpr}
   \frac{d n_1 }{ dt } =&  - \frac{1+\tau}{\tau} n_1, \nonumber \\
    \frac{d n_k }{ dt } =& n_{k-1}  - \frac{1+\tau}{\tau} n_k, \quad k = 2,  \dots, N \\
    \frac{d R }{ dt } =&  n_N  \nonumber 
\end{align}
with initial condition $n(0) = e_1 \in \mathbb R^{N}$ and $R(0)=0$.

Since the kinetic proofreading mechanism that we introduced above can have only two possible outcomes (response/non-response), we introduce a random variable $X$, with a Bernoulli distribution, taking only two possible values, $1$ and $0$, hence 
\[
\mathbb P (X=1 )= p(\tau) =  1 - \mathbb P(X=0).  
\]
In particular $\mathbb P (X=1) $ is the probability of having a response after one kinetic proofreading trial and $1- \mathbb P (X=1)$ is the probability of non-response. 
Therefore 
\begin{equation} \label{p as a function of R}
p (\tau) = \int_0^\infty n_N(t) dt . 
\end{equation} 
To compute the probability of response $p (\tau)$, we compute the solution of the system of ODEs \eqref{ODE kpr}. 

\begin{lemma} \label{lem:solution ODE} 
Let $\tau >0 $ and $N \geq 1 $ and assume that $n(0)=e_1 \in \mathbb R^N$ and $R(0)=0$. 
The solution of \eqref{ODE kpr} is $(n(t), R(t))=((n_1(t), \dots n_N(t)), R(t)) $ where 
\[ 
n_1(t)=e^{- t \left( \frac{1 + \tau }{  \tau} \right) }, \quad t >0 
\] 
while 
\begin{equation} \label{n_i} 
n_k(t)=   \frac{t^{k-1}}{(k-1)!}  e^{- t \left(  \frac{1 + \tau}{\tau} \right) }, \quad k= 2, \dots N,  \quad t >0
\end{equation}
and
\begin{equation}\label{R}
R(t) = \int_0^t n_N(s) ds \quad t >0. 
\end{equation}
\end{lemma}
\begin{proof}
    Since  $ \frac{d n_1 }{ dt } =  - \left( \frac{1+\tau}{ \tau} \right) n_1$ and $n(0)=e_1$ we immediately deduce that $n_1(t)=e^{- t \left(  \frac{1+\tau}{  \tau} \right) }$. 
We conclude the proof by induction. 
Indeed 
\begin{align*}
    n_2(t)= \int_0^t n_1(s) e^{(s-t) (\frac{1+\tau }{ \tau })} ds =  te^{- t \left( \frac{1 + \tau }{ \tau} \right) }. 
\end{align*}
Hence \eqref{n_i} holds for $k=2$. 
Assume now that for some $ k \in \{ 2, \dots, N \}  $ it holds that $n_k$ is given by \eqref{n_i}. Then $n_{k+1} (t) $ is given also by \eqref{n_i}. Indeed
\begin{align}
      n_{k+1}(t)=  \int_0^t n_k(s) e^{(s-t) ( \frac{1+\tau }{\tau })} ds =\frac{1}{(k-1)!} e^{- t \left( \frac{1+\tau }{ \tau} \right) } \int_0^t s^{k-1} ds=\frac{1}{k!} t^k e^{- t \left(  \frac{1+\tau }{ \tau} \right) }. 
\end{align}
      As a consequence \eqref{n_i} holds for every $k \in \{2, \dots, N\} $. 
      Finally \eqref{R} follows just by the fact that $\frac{d}{dt} R = n_N$. 
\end{proof}

As a consequence of Lemma \ref{lem:solution ODE} we can compute the probability $p(\tau)$ of having a response after one kinetic proofreading trial. 

\begin{proposition}
Assume $\tau >0$ and $N \geq 1$. 
Let $(n_1(t) , n_2(t)  , \dots, n_N(t) , R(t) )$ be the solution of \eqref{ODE kpr}. 
Then the function $p $ defined as \eqref{p as a function of R} is given by \eqref{p(tau)}. 
\end{proposition} 
\begin{proof}
Notice that 
        \begin{equation} \label{n_N erlang}
    n_N(t)=\left( \frac{\tau } {1+ \tau } \right)^N E\left(t; N , \frac{1+\tau }{ \tau }\right). 
    \end{equation}
    where $E\left(t; N , \frac{1+\tau }{ \tau }\right)$ is the Erlang distribution (see \cite{RE}) of parameters $N$ and $  \frac{1+\tau}{ \tau }$, where
    \begin{equation} \label{erlang}
    E( t;  n, \theta)= \frac{  \theta^n t^{n-1}e^{- \theta n }} {(n-1) ! } \quad x , \theta  \geq 0. 
    \end{equation}
    Therefore 
 \begin{align*}
  p(\tau) =   \int_0^\infty n_N(t) dt =  \left( \frac{\tau }{ 1+\tau } \right)^N. 
 \end{align*}   
\end{proof}

\section{Specificity} \label{sec:specificity}
In the previous section we computed the probability $p(\tau)$ that a ligand, with unbinding time $\tau$, produces a response after one kinetic proofreading trial, $M=1$.
This probability depends on the length of the kinetic proofreading chain, i.e.~on the number $N$ of kinetic proofreading steps, and is given by \eqref{p(tau)}. 
 We refer to Section \ref{sec:delay} for the analysis of the asymptotics of $p(\tau)$ as $N \to \infty$ and for a comparison with a model with a delay. 
 
In this section, we study the probability that a ligand ($L=1$) produces a response $R$, in $M > 1$ kinetic proofreading trials. 
As anticipated in Section \ref{sec:model}, since each trial is assumed to be independent on the other, the probability that one ligand does not trigger response after $M$ kinetic proofreading trials is 
$(1-p(\tau) )^{ M}$ and therefore, since only two possible outcomes (response/non-response) are possible, the probability of response is given by \eqref{p(tau)_M}

We study the asymptotic behaviour of $p_M(\tau)$ as $N $ tends to infinity. 
Notice that in this section we are assuming that the number of ligands is equal to $1$. 
In Section \ref{sec:sensitivity} we study how the results derived in this section change if the number of ligands change, hence we will study the sensitivity of the model. 

We prove that, when the number of kinetic proofreading steps $N$ is large and the number of trials $ M$ is also sufficiently large, there exists a threshold value, $\tau_c$, depending on $M $ and on $N$, such that we obtain a sharp transition from non-response to response for unbinding times in a neighborhood of $\tau_c$. 
The critical unbinding time, for a fixed $N $ and $M$,  is defined as follows. 
\begin{definition} \label{def:critimal time}
 Let $M >1 $, $N > 1 $ and let $p(\tau ) $ be given by \eqref{p(tau)}. 
 We define the critical time $\tau_c= \tau_c(M) $ as the solution of    
 \[ 
 M p(\tau_c)=1. 
\] 
\end{definition}
We remark here that the number of responses that a ligand yields in $M $ kinetic proofreading trials is distributed according to a Binomial distribution with mean $M p (\tau ) $. 
Therefore, since the critical unbinding time $\tau _c $ is defined as the solution of $ M p (\tau_c  )=1$, then $\tau_c $ is the unbinding time characterizing a ligand that produces, in average, one response in $M $ trials. 
Notice that, as it could be expected, as $M$ increases the critical value $\tau_c $ decreases. 
Finally notice that the case $M=1 $ is not included in Definition \ref{def:critimal time}. However when $M=1 $ is natural to consider $\tau_c \sim N $ as $N \to \infty $ (see Section \ref{sec: prob response M=1} for the details). 

As a consequence of \eqref{p(tau)} we deduce that $\tau_c= \tau_c(M) $ is given by 
\begin{equation} \label{tauc exp}
\tau_c=  \frac{1}{{M  }^{\frac{1}{N}}-1}=\frac{1} {\exp\left(\frac{ \log(M )}{N}\right)-1}. 
\end{equation}

We will prove that the probability that a ligand, with unbinding time $\tau$, triggers a response changes abruptly from zero (non-response) to one (response) when $\tau $ is close to $\tau_c $, more precisely 
\begin{equation} \label{transition regime}
\left| \frac{\tau}{\tau_c} -1 \right| \approx \frac{1+\tau_c}{N}. 
\end{equation}
We refer to Theorem \ref{thm: M >>1} and Corollary \ref{cor:different thicknesses} for a more precise statement. 

It is relevant to write \eqref{transition regime} in terms of the binding energies. Indeed let $\tau = e^{E} $ and $\tau_c = e^{E_c}$, then the transition between response and non-response takes place in the region 
\[
\frac{1+\tau_c}{N}\approx \left| \frac{\tau}{\tau_c} -1 \right| = \left| e^{E- E_c  } -1 \right| \approx  \left| E- E_c \right|. 
\] 
In other words, the network is able to discriminate between a ligand that leads to response from a ligand that does not lead to response only if the difference between the binding energy of the ligand with the receptor, $E$, and the critical binding energy $E_c$ (corresponding to the critical unbinding time $\tau_c$) is such that \eqref{uncertanty principle} holds.

\bigskip 

\textbf{Probability of response for the KPR model when $ M \approx e^{bN} $ as $N \to \infty$}

\bigskip

We explain here the mean features of the asymptotic behaviour of the probability of response $p_M $ when $M \approx e^{bN} $ as $N \to \infty $, without entering into the details of the proofs, that will be written in detail later in Section \ref{sec:spec M infty}.
Indeed, our aim here is only to illustrate in a heuristic manner that the detailed asymptotic behavior of $p_{M}(\tau)$ when $\lim_{N \to \infty } \log(M) =b>0$, yields results on the specificity of the model.
Most importantly we show that the assumption $M\approx e^{bN } $ as $N \to \infty$ yield a very high specificity.
It is worth to give some numerical values to illustrate that the order of magnitude of the numbers involved is not unreasonable, in spite of the exponential dependence of $M $ on $N$. 
If $N \approx 10 $ and $b =1/2 $, then $M \approx e^{5} \approx 10^2$, hence a ligand should attach approximately $100$ times to the receptor in order to produce a response. 
If, instead, $b = 1 $ then $M \approx e^{10} \approx 10^5 $.

First of all, when $ \lim_{N \to \infty} \frac{\log(M)}{N} =b $, then the critical time $\tau_c $ is of order one, as we will prove in Proposition \ref{lem:asympt tau_c LM less than expo}. In particular 
\[
\tau_c \sim \frac{1}{e^b -1} \text{ as } N \to \infty. 
\]
We can therefore find easily the asymptotic behaviour of $p_M(\tau) $ when $\tau \approx 1 $ as $N \to \infty$. Indeed, in this case we have that since $\tau \approx 1 $, then $\lim_{N \to \infty} p(\tau)= \lim_{N \to \infty } \left( \frac{\tau}{1+\tau }\right)^N =0$. Hence using the Taylor expansion of the logarithm as well as the definition of $\tau_c$, see \eqref{tauc exp}, we deduce that
\begin{equation} \label{p_M expo}
p_M(\tau)=1- (1- p(\tau))^M = 1- e^{M  \log(1-p(\tau) )} \sim 1-e^{-M p(\tau) } = 1-e^{- \left(\frac{1+\tau_c}{\tau_c}  \frac{\tau}{1+\tau} \right)^N} \text{ as } N \to \infty.
\end{equation}
Notice that taking $\tau= \tau_c (1+\frac{x}{N}) $ for $x \in \mathbb R$ we deduce that
\begin{equation} \label{eq:expM special}
\left(\frac{1+\tau_c}{\tau_c}  \frac{\tau}{1+\tau} \right)^N \sim \left( 1 + \frac{x}{N (1+\tau_c ) }\right)^N  \sim e^{\frac{x}{1+\tau_c}} \text{ as } N \to \infty. 
\end{equation}
Using the change of variables $\xi = \frac{x }{1+\tau_c}$, we deduce by \eqref{p_M expo} and \eqref{eq:expM special} that
    \begin{equation} \label{eq:q M=ebN} 
    \end{equation} 
    for $\xi\in\mathbb R$.
We refer to Theorem \ref{thm: M >>1} for the detailed analysis of the behaviour of $p_M $ for different limits of $M$, that leads to \eqref{eq:q M=ebN} when $M \approx e^{b N} $ as $N \to \infty $.

As a consequence, due to the fact that when
\[ 
\frac{\tau}{\tau_c} -1 = \frac{\xi (1+\tau_c) }{N }, \quad \xi \in \mathbb R
\] 
the probability of response behaves like $ 1- \exp (-e^\xi ) $, we deduce that, the system has a very high specificity. 
Indeed, since $\tau_c \approx 1 $ the transition between response and non-response takes place in the region 
\begin{equation}\label{thickness exp} 
\left| \frac{\tau}{\tau_c } -1 \right| \approx \frac{ 1+\tau_c}{N} \approx \frac{1}{N} .
\end{equation}
Notice that the thickness of this region is small, i.e.~equal to $1/N$, and decreases as $N \to \infty $. 
In other words, if a ligand is characterized by the unbinding time $\tau $ with the receptor, then the system is able to discriminate between response and non-response if 
\[
\left|  \frac{\tau}{\tau_c} -1 \right| \gtrsim \frac{1}{N}.
\]

In the following sections we will study the specificity of the model under different assumptions on the behaviour of $M $ as $N \to \infty $.  In particular we prove that, as expected, the specificity of the system increases as $M $ increases. 
We start our analysis with the case $M \approx 1 $ in Section \ref{sec:spec M order 1} and then we study the different possible behaviour for the case $M \to \infty $ as $N \to \infty $ in Section \ref{sec:spec M infty}.

\subsection{Probability of response for the KPR model when $ M \approx 1 $ as $N \to \infty$} \label{sec:spec M order 1}
In this section, we study the behaviour of $p_M(\tau) $, the probability that a ligand yields a response after $M$ trials as $N \rightarrow \infty $, under the assumption that $M >1$ and that $ M $ does not dependent on $N $, hence is constant. 
To this end we firstly analyse the asymptotic behavior of the critical unbinding time $\tau_c $ as $N \rightarrow \infty $. In particular we show that $\tau_c \sim N $ as $N \to \infty$.
The case of $M = 1 $ is studied in Section \ref{sec: prob response M=1}.

\begin{proposition}[$M \approx 1$]\label{lem:asympt tau_c M order 1}
 Assume that $M$ is a constant function of $N$ and that $M>1$, then 
 \[
  \lim_{N \to \infty } \frac{ \tau_c \log(M)  }{N}  =1.
 \] 
 \end{proposition}
 \begin{proof}
     Since  $  M$ is a constant, then \eqref{tauc exp} implies the result. 
 \end{proof}
We now study the behaviour of $p_M(\tau) $, the probability that a ligand triggers the immune  response after $M$ trials as $N \rightarrow \infty $, when $M$ is a constant larger than $1$. 
As anticipated in \eqref{uncertanty principle}, in this case, the minimal difference of energies that allows the network to discriminate between two types of ligands is of order $1$. 
\begin{proposition}[$M\approx 1 $] \label{prop: M approx1}
    Assume that $M>1$ is a constant function of $N$. 
    Then 
    \begin{equation} \label{eq:q M order 1}
    \lim_{N \to \infty} p_M(\xi \tau_c) = 1- \left( 1-e^{ - \frac{\log(M)}{\xi} } \right)^{M},  
    \end{equation}
    where $\tau_c$ is given by \eqref{tauc exp} and $\xi >0$.
\end{proposition}
\begin{proof}
In this case we have that
\begin{align} \label{comp q M order 1}
q(\tau) :=1-p_M(\xi \tau_c) = (1-p(\xi \tau_c) )^{M} 
= \left( 1- \left(\frac{\xi \tau_c}{1 + \xi \tau_c} \right)^N \right)^{M}. 
\end{align}
We deduce that 
\begin{align} \label{tau expansion M order 1}
 \left(\frac{\tau_c \xi   }{1 + \tau_c \xi } \right)^N=  \left(\frac{1}{1 + \frac{1}{\tau_c \xi} } \right)^N = \exp \left( N \log \left(\frac{1}{1 + \frac{1}{\tau_c \xi} } \right)\right) \sim  \exp \left( -  \frac{N}{1 + \tau_c \xi } \right) \text{ as } N \rightarrow \infty. \nonumber
\end{align} 
Notice that in the above computation we used the fact that $\tau_c \rightarrow \infty$ as $N \rightarrow \infty $. 
Equality \eqref{comp q M order 1} implies
\begin{equation} \label{eq q final M<<N}
1-p_M(\xi \tau_c) = q( \xi \tau_c ) \sim \left(  1 - \exp \left( -  \frac{N}{1 + \tau_c \xi  } \right) \right)^{M} \text{ as } N \rightarrow \infty. 
\end{equation}
Using the fact that $\tau_c \sim \frac{N}{\log(M)} $ as $N \rightarrow \infty$ we obtain \eqref{eq:q M order 1}. 
\end{proof}

We notice that \eqref{eq:q M order 1} implies that the transition from non-response to response takes place in a wider region compared to the case $M \approx e^{bN }$ as $N \to \infty $. 
Moreover in this case the rescaling that we perform is of the form $\tau= \tau_c \xi $ with $\xi >0$. Therefore, since $\tau_c \sim N $ as $N \to \infty $, in this case the transition from non-response to response takes places in a region of order $1$, namely when $\tau $ is such that 
\[
\left| \frac{\tau}{\tau_c }- 1 \right| \approx 1 .  
\]
In other words the system has the specificity property when the binding energy $E $ of the ligand is such that 
\[
\left|E-E_c \right| \gtrsim 1 .  
\]

\subsection{Probability of response for the KPR model when $M \gg 1 $ as $N \to \infty$} \label{sec:spec M infty}
In this section we analyse the specificity of the model when the number of trials $M$ tends to infinity as the number of kinetic proofreading steps $N$ tends to infinity. 
To this end we firstly analyse the asymptotic behavior of the critical unbinding time $\tau_c $ as $N \rightarrow \infty $.
In particular, we show that we have three possible different behaviours of $\tau_c $ as $N \to \infty$ depending on whether $\log(M) \ll N $, $\log(M) \approx N $, $\log(M) \gg N $ as $N \rightarrow \infty$.

 \begin{proposition}[$M  \gg 1$]
 \label{lem:asympt tau_c M less than expo}
  Assume that $\lim_{N \to  \infty } M=  \infty $.  
\begin{enumerate}[(1)]
\item If $ \lim_{N \to \infty} \frac{\log(M)}{N }= 0  $, then $ \lim_{N \to \infty}  \frac{  \tau_c\log(M)}{N }=1 $. 
\item If $ \lim_{N \to \infty} \frac{\log(M)}{N}= b >0  $, then $ \lim_{N \to \infty} \tau_c =\frac{1}{e^b-1}$. 
\item If $ \lim_{N \to \infty} \frac{\log(M)}{N}=\infty$, then $ \lim_{N \to \infty} \tau_c M^{1/N} =1 $. 
    \end{enumerate}
    \end{proposition}
    \begin{proof}
      The result follows by \eqref{tauc exp}. 
    \end{proof}
We now study the probability of response as the number of trials tends to infinity. 
As anticipated in \eqref{uncertanty principle}, in this case, the minimal difference of energies that allows the network to discriminate between two types of ligands depends on the asymptotics of $\tau_c$.
More precisely, it is of order $1/ \log(M)$ when $\log(M) \ll N $ as $N \to \infty $. Instead, it is of order $1/ N$ when $\log(M) \gg N $ and $\log(M) \approx N$ as $N \to\infty$. 

\begin{theorem}[$ M \gg 1 $] \label{thm: M >>1}
        Assume that $\lim_{N \to  \infty } M  =  \infty  $. 
    Then the function $p_M $ defined in \eqref{p(tau)_M} satisfies
    \begin{equation} \label{eq:q M >> 1}  
\lim_{N \to \infty}   p_M \left( \tau_c + \frac{\tau_c \xi (1+\tau_c)}{N} \right)= 1- \exp \left( - e^{  \xi}   \right), 
    \end{equation} 
        for $\xi\in\mathbb R$.
\end{theorem}
\begin{proof}
    First we notice that if $\tau = (\tau_c (1+\frac{ \xi}{N} (1+\tau_c ))$, then 
    \begin{align*}
    \left( \frac{\tau }{1+\tau }\right)^N =  \left( \frac{ \tau_c \left(  1+ \frac{(1+\tau_c) \xi}{N} \right) }{1 +\tau_c \left(  1+ \frac{(1+\tau_c) \xi}{N} \right)  } \right)^N &= \left( \frac{\tau_c}{1+\tau_c} \right)^N  \left( \frac{ (1+\tau_c) \left(  1+ \frac{(1+\tau_c) \xi}{N} \right) }{1 +\tau_c \left(  1+ \frac{(1+\tau_c) \xi}{N} \right)  } \right)^N \\
    &= \left( \frac{\tau_c}{1+\tau_c} \right)^N  \left( 1+ \frac{\frac{ (1+\tau_c) \xi }{N}}{1+ \tau_c \left( 1+ \frac{ (1+\tau_c) \xi }{N} \right) }   \right)^N .  
    \end{align*}
    Therefore, since Proposition \ref{lem:asympt tau_c M less than expo} guarantees that when $\lim_{N \to \infty}  M = \infty  $ we have that $\lim_{N \to \infty } \frac{\tau_c}{N } =0$ we deduce that 
    \[ 
   \lim_{N \to \infty }    \left(\frac{\tau }{1 + \tau }\right)^N =     \lim_{N \to \infty } \left( \frac{\tau_c}{1+\tau_c} \right)^N    \lim_{N \to \infty } \left( 1+\frac{  \xi }{N}  \right)^N
    \]
    Since 
    \[
   \lim_{N \to \infty } \left( 1+\frac{ \xi }{N}  \right)^N  =  e^\xi. 
    \]
     we deduce that 
    \begin{equation}\label{exp from tau}
 \lim_{N \to \infty}   \left(  \frac{\tau}{1+\tau} \right)^N  =  e^\xi  \lim_{N\to \infty }  \left( \frac{\tau_c}{1+\tau_c} \right)^N. 
    \end{equation}
   Since  $ \lim_{N \to \infty} \left( \frac{ \tau_c }{1 + \tau_c } \right)^{N}= \lim_{N \to \infty}\frac{1}{M} =0$ we have that $ \lim_{N \to \infty}   \left(  \frac{\tau}{1+\tau} \right)^N =0.$
   
   As a consequence we arrive at the following asymptotics for $q(\tau)=1-p_M(\tau)$
     \begin{align}\label{q as exponential}
 q(\tau) &= (1-p(\tau) )^M 
= \exp \left( M \log \left( 1- \left(\frac{\tau}{1 + \tau} \right)^N \right) \right) \\
& = \exp \left( \left( \frac{1 + \tau_c }{ \tau_c } \right)^N  \log \left( 1- \left(\frac{\tau}{1 + \tau} \right)^N \right) \right) 
\sim \exp \left( - \left( \frac{1 + \tau_c }{ \tau_c } \right)^N  \left(\frac{\tau}{1 + \tau} \right)^N \right), \text{ as } N \to \infty. \nonumber  
\end{align}

 Proposition \ref{lem:asympt tau_c M less than expo}, together with the assumptions on $\tau$, imply
\begin{align*}
   \frac{1}{ \frac{1 + \tau_c }{ \tau_c }  \frac{\tau}{1 + \tau} }  = 1 + \frac{1}{(1+\tau_c ) } \left( \frac{\tau_c}{\tau}-1 \right) \rightarrow 1 \text{ as } N \rightarrow \infty. 
\end{align*}

Therefore we deduce from \eqref{q as exponential} that 
\begin{align*}
   q(\tau ) &\sim \exp \left( - \left( \frac{1 + \tau_c }{ \tau_c } \right)^N  \left(\frac{\tau}{1 + \tau} \right)^N \right) =  \exp \left( - \exp \left( N  \log \left( \frac{1}{1 + \frac{1}{(1+\tau_c ) } \left( \frac{\tau_c}{\tau}-1 \right)} \right) \right) \right)\\
   & \sim \exp \left( - \exp \left( \frac{N}{\tau_c} \frac{\tau_c-\tau}{ (1+\tau)} \right) \right) \sim \exp \left( - \exp \left( \frac{N}{\tau_c} \frac{\tau_c-\tau}{ (1+\tau_c)} \right) \right)  \sim  \exp \left( - e^{\xi}   \right) \text{ as } N \to \infty. 
\end{align*}
Then \eqref{eq:q M >> 1} follows. 
\end{proof}
We now write a corollary that follows directly by Theorem \ref{thm: M >>1} and that allows us to summarize the behaviour of $p_M$ under different assumptions on the growth of $M$. 
As we will see, the only difference between these cases is in the rescaling of $\tau $ that we consider. Indeed the rescaling depends on the critical unbinding time $\tau_c $ which has different asymptotics, as $N \to \infty$, depending on the growth of $M $ as a function of $N$. 
Since the case of $M \approx e^{ bN} $ was already discussed at the beginning of Section \ref{sec:specificity} we focus here only on $\log(M)  \ll N $ and $\log(M) \gg N $. 

\begin{corollary} \label{cor:different thicknesses}
Assume that $M$ is such that $\lim_{N \to  \infty } M=  \infty  $. 
Let $p_M $ be defined as in \eqref{p(tau)_M}. 
\begin{enumerate}
    \item If $ \lim_{N \to  \infty } \frac{\log ( M) }{N }=0$, then 
    \[
\lim_{N \to \infty}   p_M \left( \tau_c + \frac{\tau_c \xi}{\log(M)} \right)= 1- \exp \left( - e^{\xi}   \right)
    \]
    for $\xi\in\mathbb R$.
    \item If $ \lim_{N \to  \infty } \frac{\log ( M) }{N }=\infty $, then 
        \[
\lim_{N \to \infty}   p_M \left( \tau_c + \frac{\tau_c \xi}{N} \right)= 1- \exp \left( - e^{\xi} \right)
    \]
        for $\xi\in\mathbb R$.
\end{enumerate}
\end{corollary}
\begin{proof}
This corollary follows by Proposition \ref{lem:asympt tau_c M less than expo} and by Theorem \ref{thm: M >>1}. 
\end{proof}

Notice that in both cases, i.e.~both when $\lim_{N \to \infty} \frac{\log(M ) }{N} =0$ and when $\lim_{N \to \infty}  \frac{ \log( M)}{N} = \infty $, we have that the system has the specificity property.
Moreover, the probability distribution describing the transition between non-response and response is universal and the only difference between the two scaling limits is the thickness of the region of transition. 
Indeed, under both assumption we obtain that there exists a suitable rescaling of the unbinding time such that $p_M \sim 1- e^{- e^{\xi}} $ as $N \to \infty $. 
However, we have different rescalings that correspond to different thicknesses of the transition region.

Indeed, when $\lim_{N \to \infty} \frac{\log(M ) }{N} =0$ then the transition between response and non-response takes place when $\left| \frac{\tau}{\tau_c } -1 \right| \approx \frac{1}{\log( M ) } 
 $ as $N \to \infty $. 
The thickness of the region of transition from non-response to response depends only on $M$ and does not depend directly on $N$. Hence the discrimination is possible only if the unbinding $\tau $ is such that 
\[
\left| \frac{\tau}{\tau_c}-1 \right| \gtrsim \frac{1}{\log(M)} \gg \frac{1}{N} \text{ as } N \to \infty. 
\]

If, instead,  $\lim_{N \to \infty}  \frac{ \log( M)}{N} = \infty $, then Corollary \ref{cor:different thicknesses} implies that the transition between response and non-response takes place when $\left| \frac{\tau}{\tau_c } -1 \right| \approx \frac{1}{N} $. 
Therefore  also in this case the discrimination takes place when 
\[
\left| \frac{\tau}{\tau_c}-1 \right| \gtrsim \frac{1}{N} \text{ as } N \to \infty . 
\]
The specificity of the model is the same that we had in the case $M \approx e^{b N } $ as $N \to \infty $. However in this case a larger number of trials is required, indeed $\log(M) \gg N $. 
Another difference between this case and the former case ($M \approx e^{bN}$) is that in this case we have that $ \lim_{N \to \infty} \tau_c =0$.

\section{Origin of the high specificity of the KPR network} \label{sec:delay}
It has been argued sometimes that the reason behind the capability of KPR to discriminate between different ligands is due to the fact that it induces a delay in the response. 
As a matter of fact, we will see here that there is a large difference of behaviours between KPR and a model of delay. 
Indeed, in this section, we compare the probabilistic model of kinetic proofreading that we propose in this paper with a model in which the chain of phosphorylation events is substituted by a simple reaction, leading to a delay of order $N$ in the response. 
Notice that, since the characteristic time of each of the $N $ phosphorylation steps is of order $1$, this is the expected time delay for the KPR model.  

We start by considering the case in which a ligand can reattach only once to the receptor, hence we consider $M=1$.
It turns out that, for unbinding times $\tau = N x$ for $x >0$, the probability of response of the kinetic proofreading model $p(\tau)$ and the probability of response $\overline p (\tau)$ of the model with a delay differ only for the behaviour as $x \rightarrow 0^+ $. 
Indeed the probability of response of the kinetic proofreading model is $p(\tau) \sim e^{- 1/x}$ and tends to zero much faster than the probability of response  of the model with delay, $\overline p (\tau)\sim \frac{x}{1+x}$, as $x \rightarrow 0^+$ (see Figure \ref{fig2}).

We then consider the probability of response for the two models for large values of $M$.
We consider a particular scaling of $M$ and $N$ that  allows to consider before the limit as $N \to \infty $ of the probability of response in a single trial and then the limit as $M \to \infty$. 
This limit shows how the approximation $1- \exp (- e^x) $ for the probability of response after $M $ trials arises, demonstrating that the specificity of the KPR originates from the approximation $ p (\tau ) \sim e^{- \frac{N }{\tau }}$ as $N \to \infty $ of the probability that a ligand triggers a response after one trial. 
We will then illustrate that, instead, for the model with the delay the specificity does not improve when the number of trials $M$ increase. We compare the model with delay with the KPR model in Figure \ref{fig3}. 

Before presenting the model with the delay we provide the asymptotic behaviour of the probability of response when $ M =1 $ as $N \to \infty$. This will be similar to the case analysed in Section \ref{sec:spec M order 1}.

\subsection{Probability of response for the KPR model when $ M =1 $ and $N \to \infty$} \label{sec: prob response M=1}
In this section we prove that for ligands with unbinding times $\tau = N x $, the probability $p (\tau )$ of response (that is given by \eqref{p(tau)}) upon attachment of the ligand to the receptor is approximated by  $e^{ - \frac{1}{x}}$ as $N \to \infty $. 
\begin{proposition}[$M=1 $] \label{prop: M=1}
    Assume that $  M= 1 $.  
    Then 
    \begin{equation} \label{eq:q M = 1}
  \lim_{ N \to \infty }  p(x N ) =   e^{ - \frac{1}{x}},  
    \end{equation}
    for $x >0$. 
\end{proposition}
\begin{proof}
    Using \eqref{p(tau)} we deduce that 
\begin{align} \label{tau expansion M order 1}
p(x N )= &  \left(\frac{ N  x   }{1 + Nx } \right)^N = \exp \left( N \log \left(\frac{1}{1 + \frac{1}{N x} } \right)\right) 
  \sim  \exp \left( -  \frac{N}{1 + N x } \right) \sim  e^{ -  \frac{1}{ x } }\text{ as } N \rightarrow \infty. \nonumber
\end{align} 
Hence \eqref{eq:q M = 1} follows. 
\end{proof}
The most remarkable feature of this  model is the exponential behaviour of the probability of response $p(x N ) $ as $x \rightarrow 0$. As we will see later this exponential behaviour will be responsible of the high sensitivity of the KPR model (at least when $\log(M) \ll N $). 
\subsection{Comparison with a model with delay}
We compare the model of kinetic proofreading for $M=1 $ with a simple chemical model yielding a delay. 
More precisely we assume that the complex receptor-ligand can be only at two states $1$ and $ 2$. When the complex is at state $1$, the ligand can detach from the complex, with rate $1/\tau $.
Moreover we assume that the reaction
\[
(1) \rightarrow (2) 
\]
is irreversible and takes place at rate $1/N$.
When the complex reaches the state $2 $, then it produces the response with probability one. 

We can write the ODEs for the probabilities $n_1(t), n_2(t) $ of reaching the state $1$, or, respectively the state $2$, before time $t>0$. 
Then 
\begin{align*}
\frac{d}{dt} n_1(t) &= - \left( \frac{ N+\tau }{  \tau N  } \right) n_1(t) , \\
\frac{d}{dt} n_2(t) &=  \frac{1}{N} n_1(t)
\end{align*} 
with $n(0)= e_1 \in \mathbb R^2$. 

The probability that the attachment of the ligand to the receptor leads to a response is then given by 
\begin{equation}\label{prob delay}
\overline p (\tau ) = \int_0^\infty n_2(t) dt= \frac{\tau }{ \tau + N }.   
\end{equation}
\begin{proposition}[$M=1 $] \label{prop: M=1 delay} 
The function $\overline p $ defined as \eqref{overline p} satisfies 
    \begin{equation} \label{overline p}
  \lim_{ N \to \infty }  \overline p(\xi N ) =\frac{\xi }{1+\xi} ,  
    \end{equation}
    for $\xi >0$.
\end{proposition}
\begin{proof}
    Using \eqref{prob delay} we deduce that 
$
\overline p(\xi N )=  \frac{ N  \xi   }{N + N\xi } \sim  \frac{\xi }{ 1 + \xi }\text{ as } N \rightarrow \infty. 
$

\end{proof}

Notice that Proposition \ref{prop: M=1} and Proposition \ref{prop: M=1 delay} imply that $ \overline p (x N ) \sim  p (x N  ) $ as $x \rightarrow \infty$. 
Instead, the behaviours of $\overline p(x N ) $ and $p (x N   )$ are very different as $x \rightarrow 0^+$. In particular $p(x N   ) $ converges much faster than $\overline p (x N  ) $ to zero as $x \rightarrow 0^+$.
See Figure \ref{fig2}. 
\begin{figure}[H] 
\centering
\includegraphics[width=0.5\linewidth]{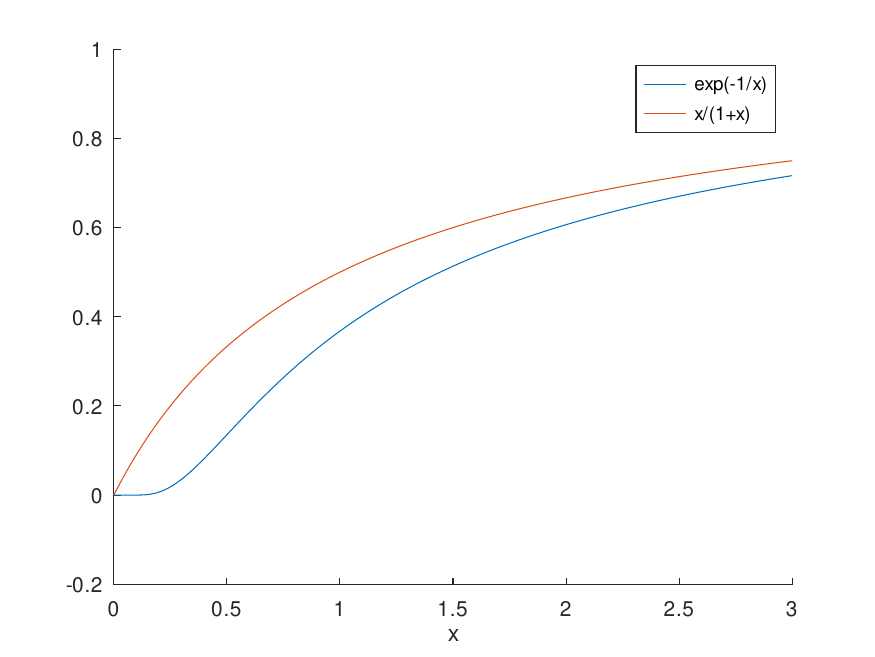}
\caption{
Plot of the probability of response, up to the scaling $\tau = x N $, of the kinetic proofreading model and of the model with a simple delay. Notice the different behaviour of the two functions at zero. 
}\label{fig2}
\end{figure}

We now consider the probability, for the model with delay, that a ligand triggers a response after having attached $M \geq 1$ to the receptor. 
Notice that the critical time $\overline{\tau}_c $ in this case is the solution of 
\[
M \overline p (\overline {\tau}_c ) =  \frac{ M \overline{\tau}_c }{\overline{\tau}_c  + N} =1 . 
\]
 Hence if $M > 1 $ we have that 
 \begin{equation} \label{overline tau_c}
 \overline \tau_c = \frac{N }{M-1 }. 
  \end{equation}

\begin{proposition}[$M $ large] \label{prop: M large delay} 
If $\lim_{N \to \infty } \frac{M}{N} =1 $, then $ \lim_{N \to \infty } \frac{M}{N}\overline \tau_c =1 $.
    If, instead $\lim_{N \to \infty } \frac{\log(M)}{N} =b>0 $, then $ \lim_{N \to \infty } \frac{  (e^{bN} -1)\overline \tau_c }{N} = 1$.
Moreover, in both the two cases, the function $\overline p_M$ defined as 
\[
\overline p_M(\tau):=1- \left( 1-\overline p (\tau) \right)^M
\]
is such that  
  \begin{equation} \label{overline p}
\lim_{N  \to \infty } \overline p_M(x \overline \tau_c ) =1-  e^{-x}
    \end{equation}
    for $x >0$. 
\end{proposition}
\begin{proof}
The asymptotics of $\tau_c$ as $N \to \infty $ follows by \eqref{overline tau_c}.  
Moreover, in both cases we obtain that 
\begin{align*}
   \overline p_M( \tau)= 1 - \left( 1 - \frac{\tau}{\tau+ N }\right)^M = 1 - \left(\frac{N}{ \tau + N } \right)^M  = 1- \left( \frac{1}{1+\frac{\tau}{N} }\right)^M \sim 1- e^{ - \frac{M}{N}\tau},  
\end{align*}
where $\tau = x \overline \tau_c $. 
In both cases we obtain \eqref{overline p}. 
\end{proof}
We now examine a particular scaling of $M $ and of $N $, with $M \to \infty $ in order to see the different discrimination properties of the stochastic KPR model and the delay model described above. 
In particular, we compare the behavior of $\overline p_M $ with the behaviour of $p_M$ as $M \approx N $ as $ N \to \infty $. 
If $M$ does not contain an exponential dependence on $N $ (for instance if $M \approx N $), it is possible to compute first the probability of response $p_M $  for $ N \to \infty $ in a single trial and later to consider the limit as the number of trials $M $ tends to infinity. 
In this way we deduce, using \eqref{eq:q M = 1}, that  $p_M (\tau ) \sim 1 - (1-e^{-\frac{N}{\tau }})^M$ as $N \to \infty$. 
Then taking the limit as $M \to \infty $, where we assume that $\log(M) \ll N $, we deduce that 
\[
p_M (\tau ) \sim 1 - (1-e^{-\frac{N}{\tau }})^M \sim 1- \exp(-M e^{-\frac{N}{\tau} }) = 1- \exp(- e^{ \log(M)-\frac{N}{\tau} }). 
\]
Since, as obtained also in Section \ref{sec:spec M infty}, the critical unbinding time is such that $\tau_c \sim \frac{N }{ \log(M) }$ as $N \to \infty$, we obtain that
\[
p_M (\tau )\sim  1- \exp(- e^{ N \frac{\tau-\tau_c }{\tau \tau_c }}) = 1- \exp(- e^{  \frac{\log(M) (\tau-\tau_c) }{\tau }}) \text{ as } N \to \infty \text{ and } M \to \infty. 
\]
Taking $ 1- \frac{\tau_c }{\tau  } = \frac{ x }{ \log(M) }  $ for $x \in \mathbb R$ we obtain the asymptotics $p_M (\tau ) \sim 1- \exp (- e^x ) $ as $N \to  \infty $ and $M\to  \infty $. 
\begin{figure}[H] 
\centering
\includegraphics[width=0.5\linewidth]{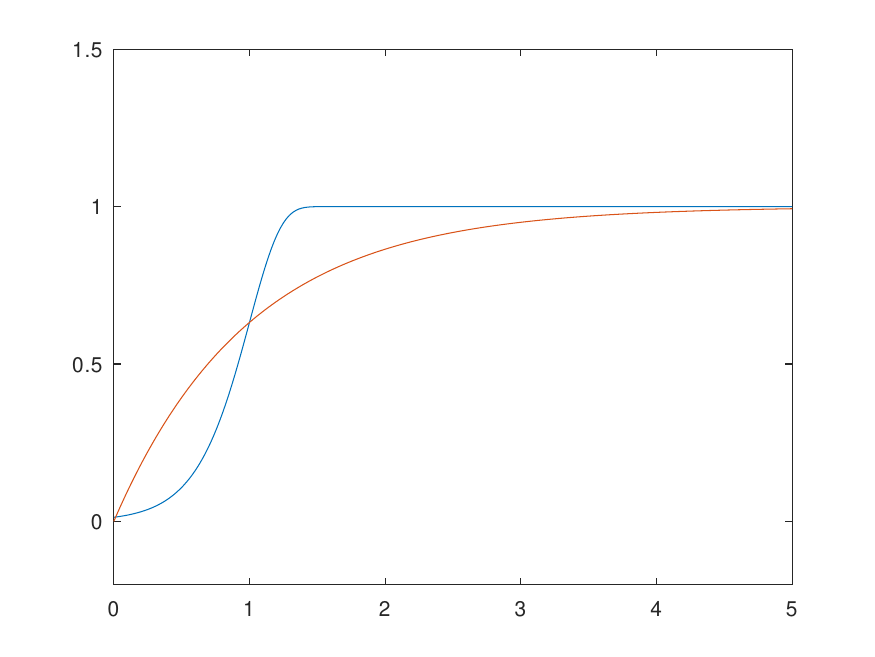}
\caption{In red we plot the approximation of $ x \mapsto \overline p_M(x \overline \tau_c)$ for $N =10$ and $M =10$.  
In blue we plot the approximation of $ x \mapsto p_M(x \tau_c)$ for the same parameters. 
}
\label{fig3}
\end{figure}

\section{Sensitivity} 
\label{sec:sensitivity}
In this section we study the sensitivity of the kinetic proofreading model. 
In the previous sections we assumed the number of ligands to be equal to one.
In this section we assume that the number of ligands $L$ is a function of the number of KPR steps, as we did for $M$. 
As a consequence, the critical unbinding time $\tau_c(L) $ could depend not only on $M $, but also on the number of ligands $L$. 
We prove that the asymptotic behaviour of the critical time $\tau_c(L) $ for $L>1$ does not differ significantly from the asymptotic of the critical time $\tau_c=\tau_c (1)$ that we had for the model with $L=1$, as long as $ \log(L) \ll N $ and $M \approx e^{ b N } $ for $b >0$. 

In particular, when $M \approx e^{bN} $ as $N \to \infty $ it turns out that, when $\log(L) \ll  N $, the critical time does not depend on the number of ligands and it depends only on the value of $b$.
Indeed, in this case, the critical time is such that 
\[
\tau_c(L) \sim \frac{1}{e^b - 1 } \text{ as } N \to \infty 
\]
as long as $\log(L) \ll N $. 
Therefore, the asymptotic behaviour of the critical unbinding time does not depend on the number of ligands when $M  \approx e^{bN } $ and $ \log(L) \ll N $ as $N \to \infty $. Therefore, under these assumptions the model satisfies both the sensitivity and the specificity properties. 
In the other cases, i.e. $M \approx 1 $, $1 \ll \log(M)  \ll N $ and $\log(M) \gg N $ and $\log(L) \ll N $ the critical time $\tau_c(L) $ depends, in some cases, on $L$. 

\subsection{Critical unbinding time for $L>1$}
We study now the asymptotic behaviour of $\tau_c $ under different assumptions on $M $ and on $L$. 
To this end we first of all define the critical unbinding time, in a similar way as we do in Section \ref{sec:sensitivity}. 
\begin{definition} \label{def:critimal time L>1}
Let $M> 1 $, $N \geq 1$ and $L> 1$. 
 We define the critical unbinding time $\tau_c(L)$ as the solution of    
 \[ 
  L M p(\tau_c(L))=1. 
\] 
\end{definition}
As a consequence, 
\begin{equation} \label{tauc exp L}
\tau_c(L)=  \frac{1}{{(LM) }^{\frac{1}{N}}-1}=\frac{1} {\exp\left(\frac{ \log(M) + \log(L)}{N}\right)-1}. 
\end{equation}

First of all we study the asymptotic behaviour of $\tau_c(L)$ as $N \rightarrow \infty $ assuming that $ \log(L) \ll N  $ as $N \rightarrow \infty $. 
We prove that if $M \approx 1$ and $L \approx 1 $ we have that $\tau_c (L) \sim \frac{N}{\log(M) + \log(L) } $ as $N \to \infty$. Hence, under these assumptions on $M $ and on $L$, the behaviour of $ \tau_c =\tau_c(1)$ and of $\tau_c(L) $, with $L>1$, differ by a constant as $N \to \infty$. However, in both cases we have that the critical unbinding time tends to infinity. 
\begin{proposition}[$M \approx 1 $and $ L\approx 1$ ]\label{lem:asympt tau_c LM order 1}
 Assume that $ L>1  $ and $M >1$ are both constant. 
Then 
 \[ 
 \lim_{N \to \infty } \frac{\tau_c(L)\left( \log(M) + \log(L)\right) }{N} =1
 \] 
 \end{proposition}
 \begin{proof}
The statement follows by \eqref{tauc exp L}.
 \end{proof}
 
Similarly, if $M \approx 1$ and $\log(L) \ll N  $ we have that $\tau_c (L) \sim \frac{N}{\log(L)} $ as $N \to \infty$. Both $\tau_c(L) $ for $L > 1 $ and $\tau_c = \tau_c(1)$ tend to infinity as $N \to \infty $, but the asymptotic behaviors are different.
\begin{proposition}[$M \approx 1 $ and $ 1 \ll \log(L) \ll N$ ]\label{lem M order 1 and log(L) ll N}
 Assume that $\lim_{N \to \infty } L =\infty $ and  $ \lim_{N \to \infty} \frac{\log(L)}{N}=0$ and $M >1$ and constant. 
Then 
 \[ 
 \lim_{N \to \infty } \frac{\tau_c(L)\left( \log(L)\right) }{N} =1
 \] 
 \end{proposition}
As a consequence, when $M \approx 1 $, to have a substantial change in the behaviour of the critical $\tau_c(L) $ for $L > 1$, i.e.~to have either that $ \tau_c(L) $ is of order one or that $\tau_c(L) $ tends to zero as $N \to \infty$, we need to have a very large number of ligands, namely $\log(L) \approx N $ or $\log(L) \gg N $.

 In the following lemma we analyse the asymptotics of the critical time $\tau_c(L) $ for large values of $N$ when $ 1 \ll M $ and $ \log(M) \ll N$. 
 We obtain different asymptotics depending on whether $\log(M) \ll \log(L) $, or $\log(M) \approx  \log(L) $ or $\log(M) \gg \log(L) $ for large values of $N$ and of $M$. 
 In particular in all the cases we have that $\tau_c (L) $ tends to infinity, as if was for $\tau_c $. Moreover, if $\log(M) \gg \log(L) $ we have that $\tau_c(L) $ behaves exactly as $\tau_c $ for large values of $N$ and for $\log(M) \ll N$. 
 \begin{proposition}[$ 1 \ll \log(M)  \ll N $ and $ \log(L) \ll N $]
 \label{lem:asympt tau_c LM less than expo}
 Assume that $\lim_{N \to \infty } M = \infty $ and 
 $\lim_{N \to \infty}\frac{\log (L) }{N }= \lim_{N \to \infty}\frac{\log ( M) }{N }=0 $. 

\begin{enumerate}[(1)]
\item If $ \lim_{N \to \infty} \frac{\log(M)}{\log(L)}= 0  $, then $ \lim_{N \to \infty} \frac{ \tau_c(L) \log(L) }{N}  =1 $. 
\item If $ \lim_{N \to \infty} \frac{\log(M)}{\log(L)}= c >0  $, then $ \lim_{N \to \infty}  \frac{(c+1) \log(L) \tau_c(L) } {N} =1  $. 
\item If $ \lim_{N \to \infty} \frac{\log(M)}{\log(L)}=\infty$, then $  \lim_{N \to \infty} \frac{\log(M) \tau_c(L) }{N}  =1 $. 
    \end{enumerate}
    \end{proposition}
    \begin{proof}
The statement follows by \eqref{tauc exp L}. 
    \end{proof}

Finally, if $\log(M) \approx N $ and $\log(L) \ll N $ we obtain that $\tau_c =\tau_c(1) \sim \tau_c(L) $ as $N \rightarrow \infty$, which means that the asymptotic behaviour as $N \to \infty $ of $\tau_c $ and $\tau_c(L) $ for $L > 1 $ are the same.  
To have that the critical time changes behaviour (i.e.~$\tau_c(L) \rightarrow 0$ or $\tau_c(L) < \tau_c $ as $N \to \infty$) we need to consider a number of ligands such that $\log(L) \gg N$.
Notice that if $N \approx 10 $, and if $M \approx e^{ 10 b } $ for some constant $b>0$, then the sensitivity property holds as long as $\log(L) \ll 10 $, i.e.~as long as $L \ll 1000$.  
  \begin{proposition} \label{lem:asympt tau_c LM expo}
  Assume that $\lim_{N \to \infty } M = \infty $ and $ \lim_{N \to  \infty } \frac{\log (M) }{N }  = b >0$ and that $\lim_{N \to \infty} \frac{\log(L) }{ N } =0$. 
  Then $ \lim_{N \to \infty} \tau_c(L ) = \frac{1}{e^b -1 } $. 
  \end{proposition}

We obtain a similar behavior when $\log(M) \gg N $. In this case we consider all the possible behaviours of $L$ and in all the cases we obtain that $ \lim_{N \to \infty} \tau_c(L) =0$
\begin{proposition}\label{lem:asympt tau_c LM more than expo}
   Assume that $ \lim_{N \to  \infty } \frac{\log ( M) }{N } = \infty $. 
   \begin{enumerate}[(1)]
\item If $ \lim_{N \to \infty} \frac{\log(M)}{\log(L)}= 0  $, then $ \lim_{N \to \infty} \tau_c(L)  L^{1/N} =1 $ . 
\item If $ \lim_{N \to \infty} \frac{\log(M)}{\log(L)}= d >0  $, then $ \lim_{N \to \infty} (d+1) L^{1/N}\tau_c (L) =1$. 
\item If $ \lim_{N \to \infty} \frac{\log(M)}{\log(L)}=\infty$, then $ \lim_{N \to \infty} {M}^{1/N}\tau_c(L) =1 $. 
    \end{enumerate}

\end{proposition}
\begin{proof}
Since $ \lim_{N \to  \infty }  \frac{\log (L M) }{N } =\infty $, then \eqref{tauc exp L} implies 
\[ 
\tau_c\sim  \frac{1} {\exp\left(\frac{ \log({M})+ \log(L)}{N}\right)} \text{ as } N \to \infty. 
\]
The desired conclusion follows by considering the dominant term in the three cases. 
\end{proof}

\section{Energy consumption} \label{sec:energy}
In this section we compute the amount of molecules of ATP consumed by $ M $ kinetic proofreading trials when $L=1$. 
As in the previous sections, we consider separately the case of $M=1$, $M \approx 1 $ but $M >1 $, and of $M \gg 1 $ as $N \to \infty $. 
The number of molecules of ATP that are consumed will be described by a random variable. We compute the asymptotics of its probability distribution as $N \to \infty $ and we obtain a mixed distribution when $M \approx 1 $ and when $M=1 $. 
In particular, when $M \approx 1 $ and $M > 1 $ this distribution consists of a sum of a Dirac in $M$ with a sum of continuous distributions on a subinterval of $(0, M)$, see Figure \ref{fig4} for a plot of this distribution. 
Instead, when $M \to \infty $ the probability distribution of the random variable, describing the amount of ATP molecules consumed, will be approximated by a Gaussian distribution with mean $ \tau  M $ and variance $ \sqrt{ \tau(1+\tau) M}  $. 
\begin{figure}[H] 
\centering
\includegraphics[width=0.7\linewidth]{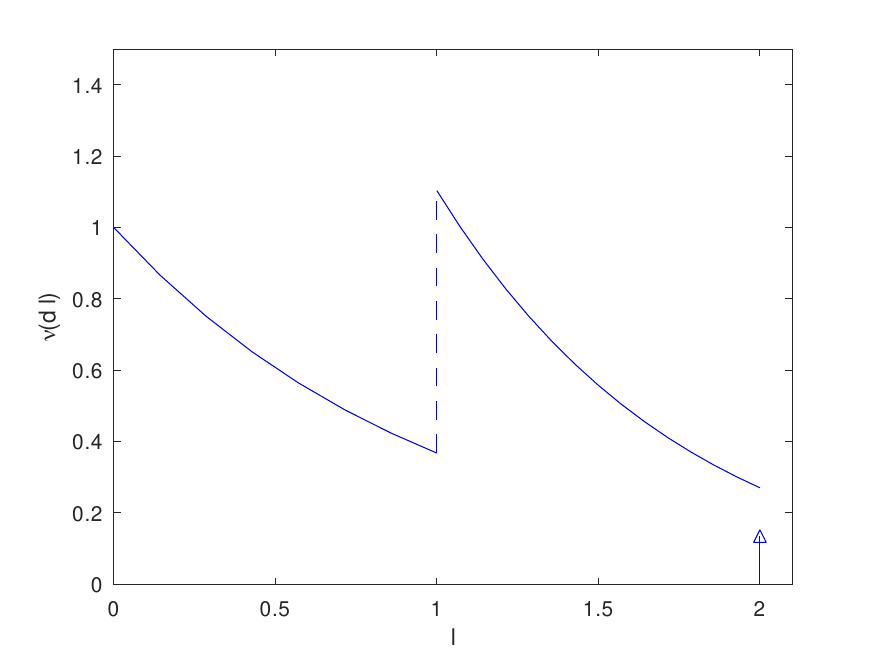}
\caption{ Here is the plot of \eqref{eq:mixed density M >1} when $M=2 $ and $x=1$. In the figure we use the notation $ \uparrow$ to indicate the Dirac in $\ell=2$ weighted by $e^{-2}$. 
}
\label{fig4}
\end{figure}

\subsection{Energy consumption during one kinetic proofreading trial}
In order to keep track of the number of molecules of ATP that are consumed in one kinetic proofreading trial, we introduce the variables $\{ m_k(t) \}_{k=0}^{N}$, describing the probability that a ligand detaches from the receptor after $k$ phosphorylation events in the time interval $(0,t)$, hence $k$ ATP molecules are consumed by the ligand.

By their interpretations the functions $\{ m_k\}_{k=0}^{N}$ satisfy the following system of ODEs 
\begin{align} \label{ODE consumption of ATP}
    \frac{d m_k}{dt} (t) &= \frac{1}{\tau } n_{k+1}(t), \quad \text{ for every }  k \in \{ 0, \dots, N-1 \}, \quad t >0  
\end{align}
and $m_N(t)= R(t)$, with initial conditions $m_k(0)=0$ for every $k \in \{0, \dots, N \} $ and where $n_k$ are the solutions of \eqref{ODE kpr} with initial condition $n_0 = e_1 $. 
\begin{lemma} 
Let $\tau >0$, $N \geq 1$. 
Let $ \{ m_k (t) \}_{k=1}^N $ be the solution of \eqref{ODE consumption of ATP}. Then  
    \begin{equation} \label{m_k}
 m_k(\infty ) := \lim_{ t \rightarrow \infty } m_k(t)= \frac{1}{1+\tau } \left( \frac{\tau }{1+\tau }\right)^k, \text{ for every } k \in \{0, \dots, N -1 \}. 
\end{equation}
Moreover 
\[
m_N(\infty):=  \lim_{ t \rightarrow \infty } R (t)= \left( \frac{\tau }{1+\tau } \right)^N. 
\]
\end{lemma}
\begin{proof}

For every $k \in \{ 0, \dots,  N-1\} $ we have that 
\[
\hat{m}_k(z) = \frac{1}{\tau} \hat{n}_{k+1}(z)= \frac{1}{\tau z } \left( \frac{1}{1+z +\frac{1}{\tau}}\right)^{k+1} \text{ for }  z \in \mathbb C,
\]
where $\hat{m}_k$ is the Laplace transform (see \cite{feller1991introduction} for the definition of Laplace transform) of $m_k$ and $\hat{n}_k$ is the Laplace transform of $n_k$. 

Using the final value theorem for Laplace transforms (\cite{feller1991introduction}) we deduce that for every $k \in \{0, \dots, N-1 \} $
\begin{align*} 
m_k(\infty)= \lim_{ t \rightarrow \infty } m_k(t)= \lim_{ z \rightarrow 0 } z \hat{m}_k(z) = 
 \frac{1}{1+\tau } \left( \frac{\tau }{1+\tau }\right)^k. 
\end{align*} 
Moreover, 
\[
R(\infty)=  \lim_{ t \rightarrow \infty } R (t)=\int_0^\infty n_N(t) dt =  \left( \frac{\tau }{1+\tau } \right)^N. 
\]
\end{proof}

Notice that $ \{ m_k (\infty) \}_{k=0}^{N}$ are the probabilities of consuming $k $ molecules of ATP during one kinetic proofreading trial with $N$ steps. 
We want now to compute the asymptotics of the probabilities $ \{ m_k (\infty) \}_{k=0}^{N}$ as $N \to \infty$.  

\begin{theorem} \label{thm: energy cons M=1}
    Let $\{ m_k (\infty) \}_{k=0}^N $ be given by \eqref{m_k} as a function of $\tau= N x $ for $x >0$. 
    Then
    \[
    \mu_N  \rightharpoonup \mu  \text{ as } N \rightarrow \infty
    \]
    where $\mu_N $ and $\mu$ are the probability measures on $[0,1]$
    \[ 
    \mu_N (d \ell ) :=  \sum_{k=0}^N m_k (\infty) \delta\left(\ell - \frac{k}{N} \right) 
    \] 
     for every $N \geq 1 $ and 
     \begin{equation} \label{eq:mixed density M=1}
    \mu (d \ell) :=  \frac{e^{- \frac{\ell }{x} }}{x} \chi_{[0,1) } (  \ell) d \ell + e^{- \frac{\ell }{x} }  \delta (\ell -1)
     \end{equation}
\end{theorem}
  For the proof of this theorem we refer to Appendix \ref{app:energy M=1}. 
  Notice that Theorem \ref{thm: energy cons M=1} has a simple interpretation. 
  If the attachment of a ligand to the receptor triggers a response, then the number of molecules of ATP spent by the KPR network is equal to $N $. 
  On the other hand, if the attachment of the ligand to the receptor does not yield a response, then the distribution of $k/N $ converges to a Poisson distribution restricted to $ k /N < 1$. 
  
\subsection{Energy consumption when $M$ is of order one and $N \to \infty$}
We now aim at computing the probabilities $ \{ Q_k \}_{k=0}^{NM} $ of producing $k $ molecules of ATP in $M>1 $ kinetic proofreading trials, for large $N$. 
In this section, we study the case in which $M \approx 1 $ as $N \to \infty $. 

Let $\mathcal E_N$ be the random variable defined as 
\begin{equation} \label{random variable ATP consumption}
\mathcal E_N= \sum_{\ell=1 }^M k_\ell,
\end{equation}
where $ \{ k_\ell \}_{\ell=1}^M  $ is a set of independent random variables. 
In particular, the random variables $k_\ell$ have probability distribution $ \{ m_j(\infty) \}_{j=1}^{N}$.
The random variable $\mathcal E_N $, then describes the total amount of ATPs used during $M $ trials. The probability distribution of $\mathcal E_N$ is 
\begin{equation} \label{Q_k}
Q_k  = \left( m_1(\infty) * m_2(\infty ) * \dots *  m_M(\infty ) \right) [k] 
\end{equation}
where 
\[
(m_1(\infty) * m_2(\infty )) [k] := \sum_{j=0}^k m_j (\infty) m_{k-i}(\infty ). 
\]

\begin{theorem} \label{thm:energy M>1}
 Let $\{ m_k (\infty) \}_{k=0}^{N} $ be given by \eqref{m_k} as a function of $\tau= N x $ for $x >0$. 
Assume that $  M >1$ is constant. 
    Then
    \[
    \nu_N  \rightharpoonup \nu \text{ as } N \rightarrow \infty
    \]
    where $\nu_N$ is given by 
       \begin{equation} \label{nu_N}
    \nu_N (d \ell ) :=  \sum_{k=0}^N  Q_k \delta\left(\ell - \frac{k}{N} \right). 
    \end{equation}
    where $ \{ Q_k\}_{k=1}^{MN} $ are given by \eqref{Q_k} as a function of $\{ m_k (\infty) \}_{k=0}^{N} $  and  
     \begin{equation}\label{eq:mixed density M >1}
    \nu (d \ell) :=  e^{- \frac{M}{x}} \delta(\ell - M )+  \frac{ e^{- \frac{\ell}{x}}}{x} \sum_{k=1}^M \binom{M}{k}  \sum_{ j=0}^k \binom{k }{j} (-1)^{k-j+1}  \frac{1}{(k-1)!} \left( \frac{\ell - M  + j }{x} \right)^{k-1} \chi_{[0, M-j ]  } ( \ell)  d \ell . 
     \end{equation}
\end{theorem}
Notice that, for $M=1$ we have that \eqref{eq:mixed density M >1} is equal to \eqref{eq:mixed density M=1}. For the proof of this theorem we refer to Appendix \ref{app:energy M>1}.

\subsection{Energy consumption when $M\to \infty$ as $N \to \infty$} 

We aim now at computing the probabilities $ \{ Q_k \}_{k=0}^N $ of producing $k $ molecules of ATP in $M>1 $ kinetic proofreading trials, for large $N$. 
In this section, we study the case in which $M \to \infty $ as $N \to \infty $ and we study the case in which the unbinding time is of the form $   \tau= \tau_c \left( 1+ \frac{x}{N} (1+\tau_c) \right)$ for some constant $x$. 
In particular we will prove the following theorem. 
\begin{theorem} \label{thm:gaussian ATP}
Let $\{ m_k (\infty) \}_{k=0}^N $ be given by \eqref{m_k} as a function of 
\begin{equation} \label{tau M to infty}
    \tau= \tau_c \left( 1+ \frac{x}{N} (1+\tau_c) \right), 
    \end{equation}
    where $ x\in \mathbb R$ and $\tau_c $ is given by \eqref{tauc exp}. 
Assume that $\lim_{N \to \infty} M = \infty $.
Then the random variable $\mathcal E_N $ defined in \eqref{random variable ATP consumption} is such that 
    \begin{equation}\label{normal approx energy}
    \frac{\mathcal E_N - \tau M }{ \sqrt{ \tau (1+\tau) M} } \overset{d}{\rightarrow}\mathcal N (0,1)  \text{ as } N \rightarrow \infty, 
    \end{equation}
    where $\mathcal N (0,1)$ is the normal distribution with mean $0 $ and standard deviation $1$. 
\end{theorem}
Notice that we are using the notation $\overset{d}{\rightarrow} $ to refer to the convergence in distribution. 
  For the proof of this theorem we refer to Appendix \ref{app:energy M>>1}.

\section{Speed}
\label{sec:time} 
In this section we study the speed of the kinetic proofreading model introduced in Section \ref{sec:model}. 
In particular we study $T_{receptor} $, that is the total quantity of time that a ligand spends attached to the receptor in order to produce a response. In this computation we neglect the time that the ligand takes to attach to the receptor. 
We start by analysing the case of $M=1 $ in Section \ref{sec:time M=1} and later we restrict to the case of $M \approx e^{ b N } $ as $N \to \infty $, see Section \ref{sec:time M exp}. 

\subsection{Speed of one kinetic proofreading trial ($M=1$)}
\label{sec:time M=1}
We recall that $ R(t) $ is the probability of having a response in the time interval $(0, t) $, see Section \ref{sec:model}. Therefore the probability of having a response at time $ t >0 $ is given by $\frac{d}{dt} R(t) = n_N(t)$, because $R(t) $ is given by \eqref{R}.  
We compute the conditional probability density $\Psi_R $ of having a response at time $t>0 $, conditioned to the fact that the ligand triggers a reaction, which is given by equality \eqref{p(tau)}. 
Then the definition of conditional probability implies that 
\begin{align} \label{psiR}
\Psi_R(t) =  \frac{ \frac{d}{dt} R(t) }{ p(\tau) } =   n_N(t) \left(\frac{1+\tau }{ \tau } \right)^N  . 
\end{align} 
In the following Lemma we prove that, if $N$ is fixed, then $\Psi_R $ is an Erlang distribution of parameters $N$ and $\frac{1+\tau}{\tau }$.
\begin{lemma} \label{lem:KPR time}
Let $N > 1$ and $\tau >0$. 
The conditional probability density $ \Psi_R(t)$ defined as \eqref{psiR} is an Erlang distribution of parameters $N$ and $\frac{1+\tau}{\tau }$. 
Namely $\Psi_R(t)= E(t; N, \frac{1 + \tau }{\tau })$ where $E$ is the Erlang probability distribution defined in \eqref{erlang} and 
\begin{equation} \label{average time}
\int_0^\infty t \Psi_R(t) dt= \frac{N \tau }{  1+\tau }, \text{ and } \int_0^\infty  \Psi_R(t) dt= 1 
\end{equation}
\end{lemma}
\begin{proof}
Since the solution $(n_1(t), \dots n_N(t), R(t) )$ of \eqref{ODE kpr} is such that $n_N$ is given by \eqref{n_N erlang}. As a consequence then we have that 
\[ \int_0^\infty n_N(t) dt = \left( \frac{\tau }{1+\tau }\right)^N .
\] 
    As a consequence 
    \[
    \Psi_R(t)  = E\left(t;  N ,  \frac{1+\tau }{ \tau }\right).
    \]
    Equality \eqref{average time} follows by the fact that the Erlang distribution of parameters $k$ and $\lambda$ has mean $k / \lambda $. 
\end{proof}

 \subsection{Speed of $M\approx e^{b N } $ kinetic proofreading trials}
 \label{sec:time M exp}
Let $\{ n_k \}_{k=1}^N$ be as in Lemma \ref{lem:solution ODE}. Then let
\begin{equation} \label{eq:s}
s(t) := \frac{1}{\tau } \sum_{k=1}^{N} n_k (t); \quad r(t):= n_N(t) 
\end{equation}
Since for each $ 1 \leq k \leq N $ we have that $n_k(t) $ is the probability that a complex reached the state $k $ in the time interval $(0,t) $ then $s(t) $ is the probability that the ligand detaches before producing $R$. 
Instead $r(t) $ is the probability of a response before time $t >0 $. 

Then the probability density of having a response in the time interval $[t, t+dt]$ conditioned to having a response is given by $\Psi_M (t) dt $ where
\begin{equation} \label{psi_M}
\Psi_M(t) := \frac{ \sum_{j=1}^M \overbrace{s(t)* \cdots * s(t) * s(t)}^{j-1 \text{ times }} * r(t) }{  p_M(\tau)  }
\end{equation}
where $p_M $ is given by \eqref{p(tau)_M}. 

We now state the following lemma describing the asymptotic behaviour of the Laplace transform of \eqref{psi_M}. 
\begin{proposition}
     \label{lem:scaled time}
     Assume that $\lim_{N \to \infty} M=\infty $ and that $\lim_{N \to \infty} \frac{\log(M) }{N}= b >0$. 
Assume that 
\[
\tau = \tau_c \left(  1+ \frac{\xi (1+\tau_c)}{N }\right)
\] 
for $\xi \in \mathbb R$. 
Then 
\begin{equation} \label{asympt psi_M}
\lim_{N \to \infty } \hat{\Psi}_M \left( \frac{\zeta }{ \tau M } \right)  = \frac{1- e^{- \zeta} e^{- e^{\xi }}}{ 1 - e^{- e^{\xi }} } \cdot \frac{e^{ \xi }}{\zeta + e^{\xi }},\ \zeta \in \mathbb C. 
\end{equation}
\end{proposition}

The conditional probability distribution $\Psi_M$, whose Laplace transform satisfies \eqref{asympt psi_M} has very different behaviours as $\xi \rightarrow \infty $ and as $\xi \rightarrow - \infty $. 
More precisely we have the following theorem 
\begin{theorem} \label{thm:time}
Assume that $\lim_{N \to \infty} M=\infty $ and that $\lim_{N \to \infty} \frac{\log(M) }{N}= b >0$. 
Assume that 
\[
\tau = \tau_c \left(  1+ \frac{\xi (1+\tau_c)}{N }\right)
\] 
for a fixed $\xi \in \mathbb R$. 
Let $t \mapsto F(t)   := \tau M \Psi_M (\tau M t ) $. 
Then 
\[
\lim_{\xi  \to - \infty } \lim_{N \to \infty } F(t)  = \chi_{[0, 1 ] } (t)
, \ \text{ a.e. } t \geq 0 \]
and $ t \mapsto G (t)  := \frac{\tau M }{e^{\xi} } \Psi_M( \frac{\tau M t }{e^\xi}  )  $
\begin{equation} 
 \lim_{ \xi \to \infty } \lim_{N \to \infty} G (t) = e^{- t }  \  \text{ a.e. } t \geq 0.  
\end{equation}

\end{theorem}
The details of the proofs of Theorem \ref{thm:time} and of Proposition \ref{lem:scaled time} can be found in the Appendix \ref{app:speed}.

Notice that the above theorem states that, when the unbinding time $\tau $ is larger than the critical unbinding time $\tau_c$ (more precisely when $\frac{\tau}{\tau_c} -1 =\frac{\xi }{N} $) the  conditional probability of having a response in the time interval $ [t, t+dt] $, conditioned to the fact that the ligand triggers a response, is exponentially distributed, with parameter $\frac{e^\xi }{\tau M }$ for large values of $N$ and $M$.
As a consequence, for large values of $N $ (and for $M \approx e^{b N } $) the average time spent by a ligand with unbinding time $\tau > \tau_c $ in the network is $ \tau M  e^{-\xi }$ where $\xi = N \left( \frac{\tau}{\tau_c}-1 \right) $.
Notice that the average time spent in the network decreases as  $\xi $ increases. 
This is consistent with the fact that ligands with higher affinity with the receptor are expected to be faster in producing a positive response of the immune system than ligands with lower affinity with the receptor.

When $(\frac{\tau}{\tau_c} -1 ) \gg  \frac{1}{N} $ as $N\to \infty $, the conditional probability of having a response in the time $[t, dt+t ]$, conditioned to the fact that the ligand triggers a response, is uniformly distributed in the interval $[0, \tau M ] $.
Therefore the average amount of time that these type of ligands spend attached to the receptor is $ \frac{\tau M } {2} $. 
We stress that this average time is much bigger than the average time spent in the network by a ligand with unbinding time satisfying $\frac{\tau}{\tau_c} -1 =\frac{\xi }{N} $ when $\xi >0 $ is large, i.e. $\tau M  e^{-\xi }$.

\section{Fluctuations of the fluxes} \label{sec:specificity k response}
One of the quantities that is usually studied for the classical formulation as a system of ODEs of the Hopfield-Ninio chemical network, in order to compare the affinities of two ligands with a receptor, is the ratio of the fluxes of product induced by the two ligands, see for instance \eqref{compare fluxes} and also the analysis in \cite{kirby2023proofreading}. 
In this section we explain how to compute these fluxes in our model.

A natural way to interpret the fluxes in the setting of the stochastic KPR model considered in this paper, is the ratio $k/M$, where $M $ is the number of trials of the KPR and $k$ is the number of responses triggered by a ligand. 
Notice that $ J (\tau)= k /M $ is a random variable, whose statistical properties can be studied. Moreover, since $k $ is a function of the unbinding time $\tau $, also the flux $J(\tau)= k/M $ depends on $\tau$. 
To this end we study the probability that a ligand yields $k \leq  M $ responses in $M $ trials. 
It has been argued in \cite{kirby2023proofreading} that the fluctuations in the fluxes would not make it possible to derive information about the binding time from the flux $J(\tau )$. 
We will see in this section that, if $M $ is chosen sufficiently large, the measurements of the fluxes $J(\tau)$ allow to obtain $\tau $ with great accuracy. 
The maximal accuracy that can be obtained for a given value of $M $ will be computed in detail, see \eqref{overlap}. 

Notice that it is not clear that any biological system uses the fluxes of response triggered by a ligand to determine $\tau $. 
It is most likely that the way of working of the KPR network, in particular to discriminate between different unbinding times, is merely related to the absence of the presence of the response. In other worlds it is more likely that the KPR  model behaves like a digital system, yielding only two possible outcomes, i.e.~response or non-response, as explained in Section \ref{sec:specificity}.

We introduce the random variable $X_M $ with probability density given by  $\{ q_k(\tau) \}_{k=1}^M$. 
For a finite number of kinetic proofreading steps $N$ we have that $q_k(\tau ) $ is a Binomial distribution, in particular 
\begin{equation} \label{binomial}
q_k(\tau) = \binom{M }{ k } (1- p (\tau ))^{M - k } p(\tau)^k, 
\end{equation}
where we recall that $p(\tau) $ is given by \eqref{p(tau)}. 
We compute the asymptotics of $q_k $ as $M \to \infty $ when $N \to \infty$. 
We obtain two different behaviours depending on whether $M p (\tau ) \rightarrow \infty $ as $N \to \infty $ or $ \lim_{N \to \infty} M p (\tau ) = c >0  $.
In the first case we obtain that $X_M $ is approximated by a Gaussian, in the second case by a Poisson distribution. 
\begin{theorem} \label{thm: asympt normal}
Assume that $\lim_{N \to \infty } M  = \infty $ and assume that $\tau$ is such that $\lim_{N \to \infty } M p(\tau)= \infty $, where $p(\tau) $ is given by \eqref{p(tau)}. 
Then 
\begin{equation} \label{eq:binom with normal general}
\frac{X_M - M p (\tau ) }{ \sqrt{M p (\tau) \left[ 1-p(\tau)\right] }}  \overset{d}{\rightarrow} \mathcal N(0,1) \text{ as } N \rightarrow \infty.
\end{equation} 
\end{theorem}
\begin{proof}
    Notice that the random variable $X_M$ follows a Binomial distribution of mean $ \mathbb E[X_M]=M p (\tau) $ and variance $\operatorname{Var}[X_M ]= Mp (\tau) \left[ 1- p (\tau) \right] $. 
    By assumption we have that $M p(\tau) \rightarrow \infty $ as $N \to \infty $ (and hence as $M \to \infty $). Then the result follows by an application of the central limit theorem (\cite{feller1991introduction}). 
    
\end{proof}

\begin{corollary}
   Assume that  $\lim_{N \to \infty} M=\infty $ and $\lim_{N \to \infty } \frac{\log( M)}{N}  = b $ where $b >0 $.
   Moreover assume that $\tau$ is such that
\[
\lim_{N \to \infty } N \frac{\tau- \tau_c }{\tau_c (1 + \tau) } = \infty,
\] 
and 
\[
\tau (1-\tau_c)\leq 2 \tau_c  \text{ when } \tau_c \neq 1  \text{ and }  \tau \leq 1 \text{ when } \tau_c =1, \text{ for every } N \geq 1
\] 
where $\tau_c$ is given by \eqref{tauc exp}.
Then 
\begin{equation}\label{eq:binom with normal}
\frac{X_M - M p (\tau ) }{ \sqrt{M p (\tau) }}  \overset{d}{\rightarrow} \mathcal N(0,1) \text{ as } N \rightarrow \infty.
\end{equation}
\end{corollary}
\begin{proof}
    First of all notice that the assumptions on $\tau $ imply that $M p (\tau) \rightarrow \infty $ as $N \to \infty $. 
    Indeed
    \begin{align*}
   M p (\tau )= \left( e^b  \left(\frac{\tau}{1+\tau}  \right)  \right)^N  = \exp \left[ N \log \left( \frac{1+\tau_c}{\tau_c } \cdot \frac{\tau }{ 1 + \tau }  \right)  \right]   
    \end{align*}
    Using the fact that $ (1-\tau_c) \tau < 2 \tau_c $ when $\tau_c \neq 1$ and that $\tau \leq 1 $ when $\tau_c = 1$, we deduce that 
    \begin{align*}
        M p (\tau) \sim \exp \left[ N \frac{\tau- \tau_c }{ \tau_c (1+\tau) } \right] \text{ as } N \to \infty. 
    \end{align*}
    Hence $\lim_{N \to \infty } M p (\tau ) = \infty  $.
    Therefore Theorem \ref{thm: asympt normal} implies that 
    \eqref{eq:binom with normal general} holds. 
 Moreover notice that under the assumptions that we have on $\tau $ we always have that $ p (\tau ) \rightarrow 0 $ as $ N \to \infty $. 
 As a consequence \eqref{eq:binom with normal general} reduces to \eqref{eq:binom with normal}.
\end{proof}
\begin{theorem}
Assume that $\lim_{N \to \infty} M=\infty $ and that $\lim_{N \to \infty } \frac{\log( M)}{N}  = b $ where $b >0 $. 

Assume that $\tau = \tau_c \left( 1 + \frac{x}{N} (1+\tau_c) \right)$ where $x \in \mathbb R$ and where $\tau_c$ is given by \eqref{tauc exp}. 
Then 
\begin{equation} \label{eq:binomial poisson}
X_M  \overset{d}{\rightarrow} \operatorname{Pois} (x)  \text{ as } N \rightarrow \infty.
\end{equation}
\end{theorem}
\begin{proof}
First of all notice that 
\begin{align*}
    M p(\tau) = \left[\frac{1+\tau_c }{\tau_c } \cdot \frac{\tau}{1+\tau } \right]^N \sim \left( 1+\frac{ x }{N}\right)^N \text{ as } N \to \infty. 
\end{align*}
Hence 
\[
 \lim_{N \to \infty } M p (\tau) =\lim_{N \to \infty } \left( 1+ \frac{x}{N} \right)^N =e^x. 
\]
Then we can approximate the binomial distribution with a Poisson distribution of parameter $\lim_{N \to \infty } M p (\tau) = x $, see \cite{feller1991introduction}. Hence
 in this case we obtain \eqref{eq:binomial poisson}. 
\end{proof}

In order to compute the flux of product induced by $M \geq 1 $ attachments of a ligand to a receptor we define the random variable 
\[ 
J(\tau) := \frac{X_M}{M  }.
\]

Notice that, under the assumptions of Theorem \ref{thm: asympt normal}, i.e.~when $ M p (\tau) \to \infty $ and $M \to \infty$, we have that 
\[
\frac{J(\tau)-p(\tau)}{ \sqrt{\frac{p(\tau)}{M}}}   \sim \mathcal N \left( 0 , 1 \right).
\]

Assume now that $N $ is fixed and that $M \to \infty $ and that $\tau$ is of order $1$, hence $Mp (\tau)  \to \infty $. 
We can compare two fluxes $J(\tau_1 ) , J(\tau_2)$  corresponding to two different detachment times $\tau_1 \approx 1 $ and $ \tau_2 \approx 1 $.
In particular, when $\tau_1 $ and $\tau_2 $ are such that 
\begin{equation} \label{eq:unc principle p}
\left| p(\tau_1) - p (\tau_2 ) \right| \gtrsim  \sqrt{ \frac{1}{M}} \max\{ \sqrt{p(\tau_1)} , \sqrt{p(\tau_2)}  \} 
\end{equation}
for large $M$, then the system satisfies the specificity property.
Indeed the two corresponding normally distributed random variables $J(\tau_1) $ and $J(\tau_2) $ are centered respectively in $p(\tau_1) $ and $p (\tau_2) $ and have variance respectively equal to $ \sqrt{\frac{p(\tau_1)}{M }} $ and to $ \sqrt{\frac{p(\tau_2)}{M }} $. 
As a consequence, the overlapping between the probability distributions of $\xi_M (\tau_1) $ and $\xi_M(\tau_2) $ is small when \eqref{overlap} holds. 
Hence, computing the fluxes of products $p(\tau_2 ) $ and $p(\tau_1) $ allows to discriminate two ligands when \eqref{eq:unc principle p} is satisfied. See also Figure \ref{fig5}. 
    
\begin{figure}[H]%
\centering
\includegraphics[width=1\linewidth]{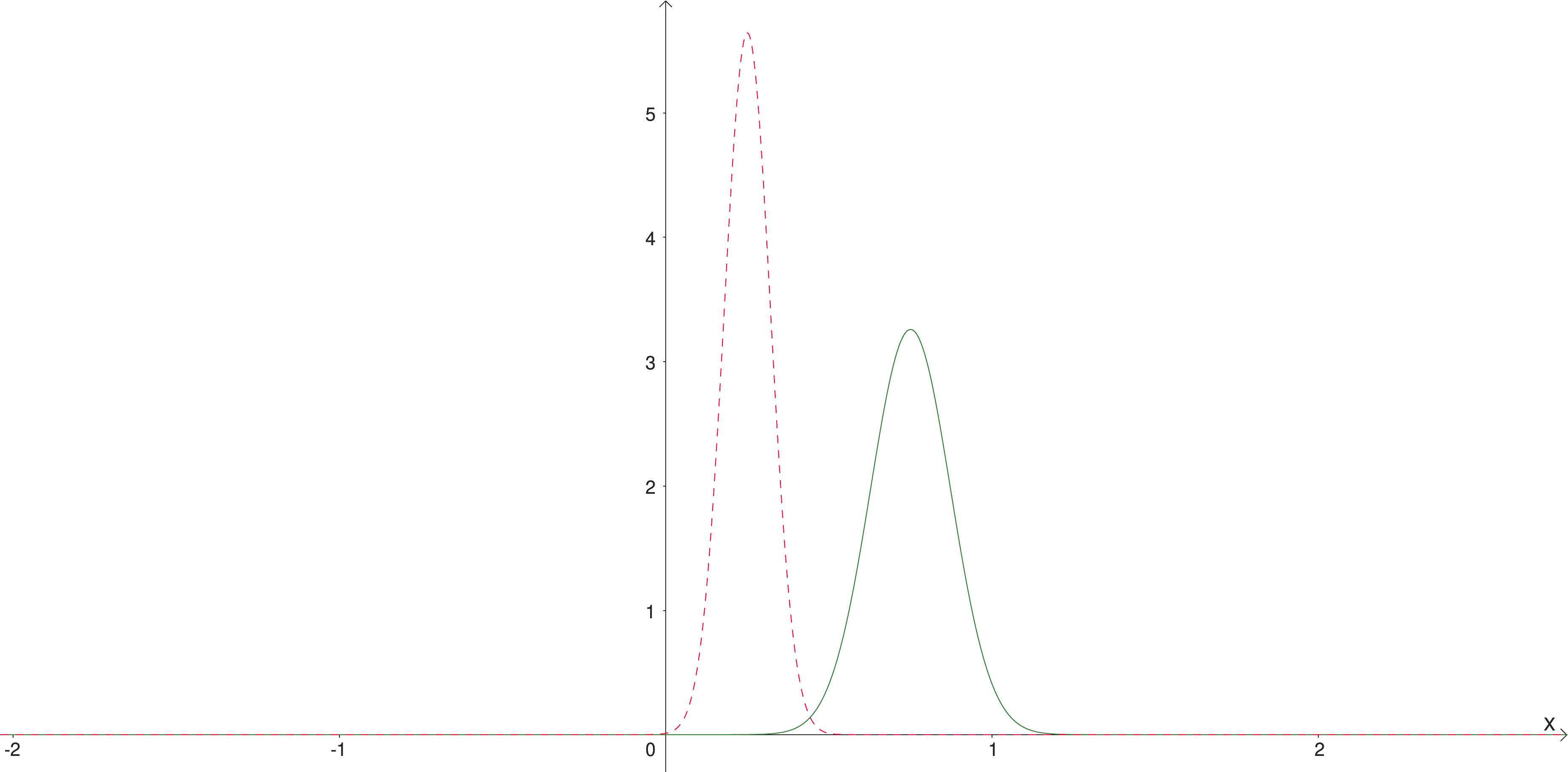}
\caption{In red with a dashed line $\mathcal N \left(p(\tau_1), \sqrt{ \frac{ p (\tau_1) } {M}}  \right)$, where $p(\tau_1 ) =1/4 $ and $M = 50$. 
In green $\mathcal N \left(p(\tau_2), \sqrt{ \frac{ p (\tau_2) }{M}} \right)$, where $p(\tau_2 ) =3/4$ and $M $ is as before. 
Notice the small region of overlap due to the fact that \eqref{eq:unc principle p} holds.
}
\label{fig5}
\end{figure}
Assume that $\tau_2 > \tau_1 $. Recall that $p (\tau ) $ is given by \eqref{p(tau)}. Hence, we can rewrite \eqref{eq:unc principle p} as 
\[
\left|  \exp \left(  \frac{  N \tau_1}{1+\tau_1 }\right) - \exp \left(   \frac{ N \tau_2}{1+\tau_2 }\right)\right| \gtrsim  \sqrt{ \frac{1 }{M} }\exp \left(  \frac{ N \tau_2}{ 2(1+\tau_2) }\right)   
.
 \]
This inequality can be rewritten on the following form 
\[
\left|  \exp \left( N \log\left( \frac{  \tau_1}{\tau_2 } \right) + \log\left( \frac{1+\tau_2 }{1+\tau_1 } \right) \right) - 1\right| \gtrsim  \sqrt{ \frac{1 }{M p(\tau_2) } }
\]
Since $\tau_1 \approx 1 $ and $\tau_2 \approx 1 $ we have that  $ \log\left( \frac{  \tau_1}{\tau_2 } \right) \sim \frac{\tau_1-\tau_2 }{\tau_2 }$ as $N \to \infty $ and $\log\left( \frac{1+\tau_2 }{1+\tau_1 } \right) \sim \frac{\tau_2 - \tau_1 }{1+\tau_1}$ as $N \to \infty$, we deduce that 
\[
\left|  \exp \left( N \left( \frac{\tau_2-\tau_1 }{\tau_2 (1+\tau_1) } \right) \right)  - 1\right| \gtrsim  \sqrt{ \frac{1 }{M p(\tau_2) } }\text{ as } N \to \infty. 
\]
Hence if we assume that $ \frac{\tau_2-\tau_1 }{\tau_2 (1+\tau_2) }  \lesssim \frac{1}{N} $, then we deduce that 
\[
\frac{\tau_2-\tau_1 }{\tau_2 }   \gtrsim  \frac{1+\tau_1}{N } \sqrt{ \frac{1 }{M p(\tau_2) } } \text{ as } N \to \infty. 
\]

\section{Model with different phosphorylation rates and unbinding times} \label{sec:different rates}
In order to show that the properties of the KPR model observed in this paper do not depend on the particular choice of the unbinding times and phosphorylation rates that we select in \eqref{ODE kpr overline t}, we discuss here a more general model in which the phosphorylation rate and the unbinding times depend on the phosphorylation state of the complex ligand-receptor. 
We will demonstrate that, when $ M \approx e^{ b N } $ as $N \to \infty $ the model does not differ in a significant way from the model with constant rates.

Let $n_k(\overline t)$ be the probability that a ligand reaches the phosphorylation state $k$ in the time interval $(0,\overline t)$ and  $R(\overline t)$ be the probability of a response in the time interval $(0, \overline t)$. 
Then, in this case we have that 
\begin{align} \label{ODE kpr overline t different coefficients}
   \frac{d n_1 }{ d \overline t } =&  - \left( \varphi_1 + \frac{1}{ \tau_1 } \right) n_1, \nonumber \\
    \frac{d n_k }{ d \overline t } =& \varphi_{k-1} n_{k-1}  - \left( \varphi_{k} + \frac{1}{\tau_{k}} \right) n_k, \quad k = 2 \dots N \\
    \frac{d R }{ d \overline t } =& \varphi_N n_N  \nonumber 
\end{align}
with initial condition $n(0) = e_1 \in \mathbb R^{N}$, where $e_1=(1, 0, \dots, 0)$ and $R(0)=0$.
With standard arguments we can compute $P$, the probability of haing a response. 

\begin{lemma} \label{lem:solution ODE different rates}
Let $N \geq 1 $ and let
Let $\{  \tau_k \}_{k=1}^N  $ and $\{ \varphi_k \}_{k=1}^N $ be such that $\tau_k >0 $ and $\varphi_k >0 $ . 
Assume that $n(0)=e_1 \in \mathbb R^N$ and $R(0)=0$. 
Then the solution $((n_1, \dots, n_N), R )$ of equation \eqref{ODE kpr overline t different coefficients} is such that 
\[
P := \int_0^\infty R(t) dt  = \prod_{k =1 }^N \left(  \frac{ \varphi_k }{ \varphi_k + \frac{1 }{ \tau_k }} \right). 
\]
\end{lemma}

Fix an unbinding time $\tau >0$. Then there exists a sequence $\{ b_k \}_{k=1}^N $ such that $ \varphi_k \tau_k = b_k \tau $ for every $k \in \{ 1, \dots, N \} $. We define the function $G_N : (0, \infty ) \mapsto (0, 1] $ as
\[
G_N (\tau ):= \prod_{k =1 }^N \left( \frac{b_k \tau }{1+ b_k \tau } \right). 
\]
Notice that $G_N (\tau )=P$. 
Moreover, by the interpretation of $R$ we have that $P $ is the probability of having a response. 
We can define as in Section \ref{sec:model} the probability $P_M$ of having a response after $M $ trials, as 
\[
P_M (\tau) := 1- \left( 1- P \right)^M = 1- \left( 1- \prod_{k =1 }^N \left( \frac{\varphi_k \tau_k }{1+ \varphi_k \tau_k } \right) \right)^M = 1- \left( 1- G_N(\tau) \right)^M .
\]
\begin{lemma}
Let $ N \geq 1$,  $\{ b_k \}_{k=1}^N $ such that $b_k >0 $ for every $k \in \{ 1, \dots, N \} $. 
    The function $G_N $ is monotone and such that
    \[
    G_N(0)=0 \text{ and } \lim_{\tau\to\infty}G_N(\tau)=1. 
    \]
\end{lemma}

In this case we define the critical time is $\tilde{\tau}_c$ is the solution of
    \begin{equation} \label{critical tau different rates}
 M P  = M G_N (\tilde{\tau}_c) = 1.
    \end{equation}
    where $M > 1 $. 
As in the model with phosphorylation and unbinding rates that do not depend on the phosphorylation state of the complex, we obtain that when $M \approx e^{ b N } $ as $N \to \infty $ then the critical time is finite. 

\begin{lemma} \label{lem:existence of critical time diff rates}
Assume that $ \lim_{N \to \infty } M = \infty $ and $ \lim_{N \to \infty } \frac{ \log(M) }{N } = b >0$. 
Assume that $\{ b_k \}_{k=1}^N $ is such that there exists a $\mu \in \mathcal M_+$ such that
\[
\frac{1}{ N } \sum_{k =1 }^N \delta (x- b_k ) \rightharpoonup \mu (dx) \text{ as } N \to \infty.
\]

 Then $  \lim_{N \to \infty} \tilde{\tau_c } = \overline T \in (0, \infty ) $ and is the solution of
 \[ 
  \int_{0}^\infty \log \left(  \frac{x \overline T }{1 + x \overline T } \right) \mu(dx) = - b. 
 \]
\end{lemma}
\begin{proof}
    First of all notice that a solution to \eqref{critical tau different rates} exists for every $N \geq 1 $ due to the monotonicity of $ \tau \mapsto G_N(\tau)  $ and the fact that $G_N(0)=0 $ and $G_N(\infty )=1 $. 
We now want to prove that $\lim_{N \to \infty } \tilde{\tau}_c $ is finite. 

Notice that \eqref{critical tau different rates} can be written as 
\[
\exp \left( \frac{1}{N} \sum_{k=1}^N  \log\left(1- \frac{1 }{1+ b_k \tilde{\tau}_c }  \right) \right) =\left( G_N ( \tilde{\tau}_c  ) \right)^{\frac{1}{N}}=\left(  \frac{1}{M } \right)^{\frac{1}{N}} = \exp \left( - \frac{ \log(M)}{N} \right). 
\]
Since we assume that $\lim_{N \to \infty } \frac{\log(M) }{N } = b $ we deduce that 
\begin{align*}
-b &= \lim_{N \to \infty }  \left( \frac{1}{N} \sum_{k=1}^N  \log\left(1- \frac{1 }{1+ b_k \tilde{\tau}_c }  \right) \right) = \lim_{N \to \infty } \frac{1}{N} \sum_{k=1}^N  \int_0^\infty    \log\left(1- \frac{1 }{1+ x \tilde{\tau}_c  }  \right) \delta (x- b_k ) dx \\
&=  \lim_{N \to \infty } \int_0^\infty    \log\left(1- \frac{1 }{1+ x \tilde{\tau}_c  }  \right)  \frac{1}{N} \sum_{k=1}^N \delta (x- b_k ) dx=  \lim_{N \to \infty } \int_0^\infty    \log\left(1- \frac{1 }{1+ x \lim_{N \to \infty} \tilde{\tau}_c  }  \right)  \mu(dx). 
\end{align*}
Now notice that since the function 
\[
\tau \mapsto    \int_0^\infty    \log\left(1- \frac{1 }{1+ x  \tau }  \right)  \mu(dx). 
\]
is monotonically increasing there exists a unique solution to 
\[ 
- b =    \int_0^\infty    \log\left(1- \frac{1 }{1+ x   \tau }  \right)  \mu(dx)
\]
which is $0 < \tau = \lim_{N \to \infty } \tilde {\tau}_c < \infty $. 
\end{proof}

\begin{proposition}
Assume that $ \lim_{N \to \infty } M = \infty $ and $ \lim_{N \to \infty } \frac{ \log(M) }{N } = b >0$.
    Assume that $\{ b_k \}_{k=1}^N $ is as in Lemma \ref{lem:existence of critical time diff rates}. 
    Then 
    \[
  \lim_{N \to \infty }  P_M \left( \tilde{\tau}_c \left( 1 +  \frac{\xi  }{N D} \right)  \right) = 1- \exp ( -  \exp (\xi)) \text{ as } N \to \infty, 
    \]
    where $\xi \in \mathbb R$ and where $ D:= \lim_{N \to \infty } \frac{1}{N} \sum_{k=1}^N   \frac{ 1}{1 + \tilde{\tau}_c b_k } = \mu([0, \infty)) $.
\end{proposition}
\begin{proof}
 The proof is similar to the one of Theorem \ref{thm: M >>1}. 
 Indeed, let $\tau = \tilde{\tau}_c \left( 1 +  \frac{x }{N } \right) $, then 
 \begin{align*}
     P_M \left( \tau  \right)  \sim 1- \left( 1- G_N(\tau )\right)^M \sim  1-\exp \left( M \log \left( 1- G_N(\tau )\right) \right) \sim 1- \exp (- M G_N (\tau ) ) \text{ as } N \to \infty. 
 \end{align*}
 Now notice that 
 \begin{align*}
     M G_N (\tau )&= \prod_{k=1}^N \left( \frac{ 1 +  \tilde{\tau}_c b_k }{\tilde{\tau}_c b_k  } \right)  \prod_{k=1}^N \left( \frac{ \tau b_k }{ 1+ \tau b_k } \right) = \prod_{k=1}^N \left( \frac{ 1 + \tilde{\tau}_c b_k }{\tilde{\tau}_c b_k  }  \right)\prod_{j=1}^N \frac{ \tilde{\tau}_c  b_j \left( 1 + \frac{x }{N }\right) }{ \left( 1+ \tilde{\tau}_c b_j \left( 1 + \frac{x}{N }\right) \right) } \\
     &= \left( 1+ \frac{x }{N }\right)^N \prod_{k=1}^N   \frac{ \left(1 + \tilde{\tau}_c b_k \right) }{\left( 1 + \tilde{\tau}_c b_k + \frac{x}{N} \tilde{\tau}_c b_k \right) } \\
     & \sim e^{x} \exp\left( \log \left( \prod_{k=1}^N \frac{ \left(1 + \tilde{\tau}_c b_k \right) }{\left( 1 + \tilde{\tau}_c b_k + \frac{x}{N} \tilde{\tau}_c b_k \right) } \right)  \right)  = e^{x} \exp\left(  \sum_{k=1}^N \log \left( \frac{ 1 + \tilde{\tau}_c b_k  }{1 + \tilde{\tau}_c b_k + \frac{x}{N} \tilde{\tau}_c b_k  } \right)  \right) \\
     & = e^{x } \exp\left(  \sum_{k=1}^N \log \left( \frac{ 1  }{1  + \frac{x}{N} \frac{ \tilde{\tau}_c b_k}{1 + \tilde{\tau}_c b_k }  } \right)  \right) \sim   \exp\left(x \left(1-  \frac{1 }{N} \sum_{k=1}^N   \frac{ \tilde{\tau}_c b_k}{1 + \tilde{\tau}_c b_k } \right) \right)
     \text{ as } N \to \infty. 
 \end{align*}
 Hence the desired conclusion follows taking $x = \frac{\xi }{1 - \frac{1}{N }\sum_{k=1}^N   \frac{ \tilde{\tau}_c b_k}{1 + \tilde{\tau}_c b_k } }$. 
 \end{proof}
\section{Conclusions}

In this paper, we study a simple probabilistic model of kinetic proofreading. In this model a ligand can bind to the receptor $M\geq 1 $ times. 
We analyse the asymptotic behaviour of the probability of having a response in $M $ trials for large values of $N$, the number of  proofreading steps, under different assumptions on the dependence of $M $ on $N$. 
From our analysis we deduce that, when $ M \approx  e^{bN} $ as $N \to \infty $, then the system satisfies both the specificity and the sensitivity property. See Table \ref{table} for a summary of our results for that case. 

More precisely when $M \approx e^{b N } $ as $N \to \infty $, we prove that, under a suitable scaling of the unbinding time, the probability of transition from non-response to response is of the form $1-e^{- e^\xi}$, see Figure \ref{fig3}. 
The thickness of the region in which the transition from non-response to response takes place is of the order of $1/N$, see \eqref{thickness exp}. 
Moreover, we also prove that, when $M \approx e^{b N } $ as $N \to \infty $, the critical unbinding time  $\tau_c$, at the boundary between non-response, and to response does not depend in a significant manner on the number of ligands in the system if $\log(L) \ll N $.
In other words the critical time does not change for $L$ changing in several orders of magnitude of $L $. 
In this sense, we prove that the system satisfies the sensitivity property. 

Finally we compute the time that a ligand spends in the network to produce a signal (that can lead to response or to non-response). 
In the case in which the ligand is characterized by an unbinding time that is smaller than the critical unbinding time  (hence we do not expect response), then we obtain that, for large $N$, the average time that a ligand spends in the network is $\frac{\tau M }{2}$. Instead, when the critical unbinding time characterizing a ligand is larger than the critical unbinding time (hence we expect response), then we obtain that, for large $N$, the time that a ligand spends in the network is $\tau M e^{-\xi }$ where $\xi= N \left( \frac{\tau }{\tau_c }-1  \right) $.
Notice that, as could be expected, ligands with larger unbinding times spend less time in the network in order to trigger a reaction of the immune system. 
We refer to Section \ref{sec:time} for precise statements. 

We also demonstrate, in Section \ref{sec:different rates}, that our results for $M \approx e^{bN } $ as $N \to \infty $ are robust, i.e.~they do not depend on the fact that we choose the phosphorylation and the unbinding rates to be independent on the phosphorylation state of the complex ligand-receptor. 

Finally, we stress that, also when for biologically reasonable values of $M$ and $N$ that satisfy $1 \ll \log(M) \ll N $ and $ \log(M) \gg N $ the KPR network satisfies the specificity and the sensitivity property.

\begin{center} \renewcommand\arraystretch{2.5}

\captionof{table}{Results for $M \approx e^{b N } $ as $N \to \infty $} 
\label{table}
\begin{tabular}{|l|l|} 
\hline 
 &  $M \approx e^{b N } $ as $N \to \infty$ \\ \hline 
Critical unbinding time &  $\tau_c \sim  \frac{1}{e^b-1} \text{ as } N \to \infty  $    
      \\ \hline 
 Probability of response & $p_M \left(\tau_c \left( 1+\frac{\xi (1+\tau_c) }{N }\right)  \right)\sim 1- e^{- e^\xi}$ as $N \to \infty $ \\
        \hline 
  Thickness of transition (non-response/response) & $ \left|\frac{\tau}{\tau_c }-1  \right| \approx \frac{1}{N}  $ \\
   \hline 
  Probability distribution of time spent in the network & $ T_{receptor} \sim \begin{cases}
       \text{Exp} \left(\frac{e^\xi }{ \tau M } \right)    &\text{ if } \left(  \frac{\tau}{\tau_c } -1 \right)  \approx \frac{\xi }{N}  \\
     \mathcal U (0, \tau M ) &\text{ if } \left( 1- \frac{\tau}{\tau_c }  \right)  \gg \frac{1}{N} 
   \end{cases} $ \\
   \hline 
     Probability distribution of ATPs consumed & $ \mathcal E_N \sim \mathcal N (\tau M , \sqrt{\tau (1+\tau ) M }) $ as $N \to \infty$\\
   \hline
\end{tabular}
\end{center}

\bigskip

{\centering{\Large{\textbf{Appendices}}}}

\bigskip

 \appendix
 \section{Energy consumption}
 In this section of the appendix we prove the main theorems stated in Section \ref{thm: energy cons M=1} . 
 \subsection{$M=1$} \label{app:energy M=1}
 \begin{proof}[Proof of Theorem \ref{thm: energy cons M=1}]
Notice that 
\begin{align*}
    \mu_N (d \ell ) &=  \sum_{k=0}^N m_k (\infty) \delta\left(\ell - \frac{k}{N} \right) = 
    \sum_{k=0}^{N-1}  m_k (\infty) \delta\left(\ell - \frac{k}{N} \right) + m_N (\infty) \delta\left(\ell - 1 \right) \\
    &= \frac{1}{\tau} \sum_{k=0}^{N-1} \left( \frac{\tau }{1+\tau} \right)^{k+1 }  \delta\left(\ell - \frac{k}{N} \right) + \left( \frac{\tau }{1+\tau} \right)^N  \delta\left(\ell - 1 \right) . 
\end{align*}
where $\tau= x N $. 
Moreover, since 
\[
\left( \frac{\tau }{1+\tau} \right)^N = \left( \frac{ x N  }{1+ x N } \right)^N \sim e^{- \frac{1}{x} } \text{ as } N \rightarrow \infty, 
\]
we deduce that 
\[
 \left( \frac{\tau }{1+\tau} \right)^N  \delta\left(\ell - 1 \right) \rightharpoonup e^{-\frac{1}{x}} \delta(\ell -1 ) \text{ as } N \to \infty . 
\]
Moreover, 
\[
 \frac{1}{\tau} \sum_{k=0}^{N-1} \left( \frac{\tau }{1+\tau} \right)^{k+1 }  \delta\left(\ell - \frac{k}{N} \right) =  \frac{1}{xN} \sum_{k=0}^{N-1} \left( \frac{x N }{1+xN}\right)^{k+1} \delta \left(\ell - \frac{k }{N} \right) \rightharpoonup \frac{e^{- \frac{ \ell }{x }}}{x} \chi_{[0, 1) } ( d \ell ) 
\text{ as } N \to \infty. 
\]
Indeed for every $f \in C_b ([0, \infty)]) $ we have that
\[
\frac{1}{N }\sum_{k=0}^{N-1} \left( \frac{x N }{1+xN}\right)^{k+1} f\left( \frac{k}{N} \right) \rightarrow \int_0^1 e^{- \frac{ \ell }{x }} f(\ell) d\ell 
\]
as $N \to \infty $. 
The desired conclusion then follows. 
\end{proof}

 \subsection{$ M >1 $} \label{app:energy M>1}

   \begin{proof}[Proof of Theorem \ref{thm:energy M>1}]
    Since $Q_k $ is defined by \eqref{Q_k} we deduce that $ \nu_N = \overbrace{\mu_N * \mu_N * \dots * \mu_N}^{M \text{ times} } $. 
As a consequence we have that 
\begin{equation} \label{convolution nu}
\hat{\nu}_N(z) = \left( \hat{\mu}_N (z)\right)^M.  
\end{equation}
Since $\mu_N \rightharpoonup \mu $ as $N \rightarrow \infty $, by the duality between convergence of Laplace transforms on compact sets and weak convergence of measures (see \cite{feller1991introduction}), we have that $\hat{\mu}_N \rightarrow \hat{\mu }$ as $N \rightarrow \infty $ on compact sets. 
Therefore using \eqref{convolution nu} we deduce that $\hat{\nu}_N=\left( \hat{\mu}_N \right)^M   \rightarrow \left( \hat{\mu} \right)^M  $ on compact sets as $N \to \infty $. 
This implies that $ \nu_N \rightharpoonup \nu $, as $N \to \infty$, where 
\[ 
\hat{\nu} (z)= \left( \hat{\mu}(z) \right)^M  =\left( \frac{1}{1+ x z } \right)^M \left( 1+ x z e^{-\frac{1}{x}(1+ z x ) }\right)^M, \quad z \in \mathbb C, \ z \neq -\frac{1}{x} .
\] 
Indeed, 
\[
\hat{\mu } (z)=\int_0^\infty e^{ - z \ell } \mu(d \ell)=\int_0^1\frac{e^{-{\ell \left( z + \frac{1}{x}\right) } }}{x} d \ell  + e^{- \left( z+ \frac{1}{x}\right)} = \frac{1}{1+xz} + \left( \frac{xz}{1+ xz}\right) e^{-\frac{1}{x} \left( 1+ xz\right)}, \quad z \in \mathbb C, \ z \neq - \frac{1}{x} . 
\]
Therefore, in order to obtain the probability measure $\nu $ we need to compute the inverse Laplace of $ \hat{\nu }=  \left( \hat{\mu} \right)^M $. 
First of all we rewrite $\hat{\nu} $ as follows 
\begin{align*}
\hat{\nu}(z) &= \left( \frac{1}{1+ x z } \right)^M \left( e^{ z + \frac{1}{x}}+ x z \right)^M e^{- \frac{M}{x}} e^{ - M z} \\
&= \left[ \left( \frac{1}{1+ x z } \right)^M \left( e^{ z + \frac{1}{x}}+ x z \right)^M - 1 \right] e^{- \frac{M}{x}} e^{  - M z}  + e^{- \frac{M}{x}}  e^{ -zM} .
\end{align*}
Now notice that 
\[
e^{- \frac{M}{x}}  e^{  - M z} = \int_0^\infty e^{- \frac{M}{x}} \delta(\ell - M ) e^{-z \ell} d\ell = e^{- \frac{M}{x}} \widehat{\delta  (\cdot - M )} (z). 
\]
Therefore since the function $z \mapsto  \left[ \left( \frac{1}{1+ x z } \right)^M \left( e^{ z + \frac{1}{x}}+ x z \right)^M - 1 \right] e^{- \frac{M}{x}} e^{  - M z} $ is analytic we have that 
\begin{align*}
    \nu ( d \ell) &= \frac{1}{2 \pi i } \lim_{T \rightarrow \infty} \int_{a-i T }^{a+ i T } \left[ \left( \frac{1}{1+ x z } \right)^M \left( e^{ z + \frac{1}{x}}+ x z \right)^M - 1 \right] e^{- \frac{M}{x}} e^{ z (\ell - M)} dz + e^{- \frac{M}{x}} \delta(\ell - M ),
\end{align*}
where we consider  $a < -  \frac{1}{x} $.

Now notice that if $ \ell > M  $, then deforming the contour as $\Re(z) \rightarrow  -\infty $, we obtain 
\begin{align*}
  &   \frac{ e^{- \frac{M}{x}} }{2 \pi i } \lim_{T \rightarrow \infty} \int_{a-i T }^{a+ i T }  e^{ z (\ell - M)}  \left[ \left( \frac{1}{1+ x z } \right)^M \left( e^{ z + \frac{1}{x}}+ x z \right)^M - 1 \right] dz =0.
  \end{align*} 
  Indeed as $\Re(z) \rightarrow - \infty $ we have that 
  \[
   e^{ z (\ell - M)}  \left[ \left( \frac{1}{1+ x z } \right)^M \left( e^{ z + \frac{1}{x}}+ x z \right)^M - 1 \right] \leq e^{z (\ell - M) } \rightarrow 0 \text{ as } \Re (z) \rightarrow - \infty .
  \]

Therefore, this together,  with the binomial expansion of $(e^{z+\frac{1}{x}} -1 + 1-xz )^M$ and of $(e^{z+\frac{1}{x}}-1)^k$ we deduce that
\begin{align*}
  &   \frac{ e^{- \frac{M}{x}} }{2 \pi i } \lim_{T \rightarrow \infty} \int_{a-i T }^{a+ i T }  e^{ z (\ell - M)}  \left[ \left( \frac{1}{1+ x z } \right)^M \left( e^{ z + \frac{1}{x}}+ x z \right)^M - 1 \right] dz
     \\
     &=  \chi_{[0, M )} ( \ell) \frac{ e^{- \frac{M}{x}} }{2 \pi i } \lim_{T \rightarrow \infty} \int_{a-i T }^{a+ i T }  e^{ z (\ell - M)}  \left[ \left( \frac{1}{1+ x z } \right)^M \left( e^{ z + \frac{1}{x}} -1 + 1 + x z \right)^M - 1 \right] dz \\
         &=  \chi_{(0, M )} ( \ell)  \frac{ e^{- \frac{M}{x}} }{2 \pi i } \lim_{T \rightarrow \infty}\sum_{k=1}^M \binom{M}{k} \int_{a-i T }^{a+ i T }  e^{ z (\ell - M)}      \left( \frac{e^{ z + \frac{1}{x}} -1 }{1+ x z }\right)^k   dz \\
         &=  \chi_{(0, M )} ( \ell)  \frac{ e^{- \frac{M}{x}} }{2 \pi i } \lim_{T \rightarrow \infty}\sum_{k=1}^M \binom{M}{k}  \sum_{j=0}^k \binom{k }{j} (-1)^{k-j} \int_{a-i T }^{a+ i T }  e^{ z (\ell - M)}  e^{ j\left(z + \frac{1}{x}\right)}    \left( \frac{ 1 }{1+ x z }\right)^k   dz  \\
         &=  e^{- \frac{M}{x}} \sum_{k=1}^M \binom{M}{k}  \sum_{j=0}^k I_{k, j } (x,z; \ell )      
\end{align*}
where 
\[
I_{k,j} (x,z; \ell):=  \chi_{(0, M )} ( \ell)  \frac{ 1 }{2 \pi i }  \binom{k }{j} (-1)^{k-j}  e^{ \frac{j}{x}}  \lim_{T \rightarrow \infty}  \int_{a-i T }^{a+ i T }  e^{ z (\ell - M+ j )}    \left( \frac{ 1 }{1+ x z }\right)^k   dz. 
\]

Now notice that if $\ell - M + j =: r<0$, then
\begin{align} \label{comp I}
 \frac{ 1 }{2 \pi i }  \lim_{T \rightarrow \infty}  \int_{a-i T }^{a+ i T }  e^{ z r }    \left( \frac{ 1 }{1+ x z }\right)^k   dz
&  = \frac{ 1 }{2 \pi i } e^{- \frac{r}{x}} \lim_{T \rightarrow \infty}  \int_{a-i T }^{a+ i T }  e^{ r \left( z+ \frac{1}{x} \right)}    \left( \frac{ 1 }{1+ x z }\right)^k   dz  \nonumber\\
 &= \frac{1}{(k-1)!} e^{- \frac{r}{x}} \left( \frac{r}{x} \right)^{k-1} \frac{1}{x}   
\end{align}
The last equality follows by  contour integration as $\Re(z) \rightarrow \infty $ around the pole $ z=- \frac{1}{x} $ of the function 
\[
 z \mapsto e^{ r \left( z+ \frac{1}{x} \right)} \left(  \frac{ 1 }{1+ x z } \right)^k 
\]
and by Cauchy's Residue Theorem. 
Indeed 
\begin{align*}
  \int_{a-i T }^{a+ i T }  e^{ z r }    \left( \frac{ 1 }{1+ x z }\right)^k   dz &=   \oint_{C}  e^{ z r }    \left( \frac{ 1 }{1+ x z }\right)^k   dz-   \int_{\Gamma_1 }  e^{ z r }    \left( \frac{ 1 }{1+ x z }\right)^k   dz -   \int_{\Gamma_2}  e^{ z r }    \left( \frac{ 1 }{1+ x z }\right)^k   dz  \\
  & - \int_{\Gamma_3}  e^{ z r }    \left( \frac{ 1 }{1+ x z }\right)^k   dz
\end{align*} 
where $\Gamma_1 $ is the line in the complex plan connecting $a+i T $ to $v + i T $ (where $v > -1/x$), $\Gamma_2 $ is the line connecting $v+ iT $ to $v- i T$ while $\Gamma_3 $ is the line connecting $v- i T $ to $ a- iT $, finally $C$ us the cycle obtain by concatenating the line from $a-i T $ to $a+ i T $ with $\Gamma_1 $, $\Gamma_2$ and $\Gamma_3$. 
Now notice that 
\[
\lim_{v \rightarrow \infty } \lim_{ T \to \infty } \int_{\Gamma_2 } e^{ z r }    \left( \frac{ 1 }{1+ x z }\right)^k   dz   =0\  \text{ and } \lim_{ T \to \infty } \int_{\Gamma_i } e^{ z r }    \left( \frac{ 1 }{1+ x z }\right)^k   dz  =0 \text{ for }  i= 1,3. 
\]
Finally notice that by Cauchy Residue Theorem we have that 
\[
\frac{1}{2 \pi i }   \oint_{C}  e^{ z r }    \left( \frac{ 1 }{1+ x z }\right)^k   dz = - \operatorname{Res}_{z=-1/x} \left( e^{ z r }    \left( \frac{ 1 }{1+ x z }\right)^k \right) =  -  \frac{1}{(k-1)!} e^{- \frac{r}{x}} \left( \frac{r}{x} \right)^{k-1} \frac{1}{x}
\]
Therefore \eqref{comp I} follows. 

On the other hand, with the same type of arguments used to derive \eqref{comp I} we have that if $r= \ell - M + j =0$, then 
\begin{align*}
 \frac{ 1 }{2 \pi i }  \lim_{T \rightarrow \infty}  \int_{a-i T }^{a+ i T }     \left( \frac{ 1 }{1+ x z }\right)^k   dz =
 \begin{cases}
     & 0 \text{ if } k >1, \\
     & - \frac{1}{x}   \text{ if } k=1 
 \end{cases}
\end{align*}
Finally  when $ r > 0 $ we have that contour deformations as $\Re(z) \rightarrow - \infty $ imply that
\begin{align*}
 \frac{ 1 }{2 \pi i }  \lim_{T \rightarrow \infty}  \int_{a-i T }^{a+ i T }  e^{ z r }    \left( \frac{ 1 }{1+ x z }\right)^k   dz=0. 
\end{align*}

As a consequence we deduce that for $ 1 \leq   k \leq M$, $ 0 \leq j \leq k $
\begin{align*}
I_{k,j} (z,x; \ell )=  \binom{k }{j} (-1)^{k-j+1}  \frac{1}{(k-1)!} e^{- \frac{\ell - M  }{x}} \left( \frac{\ell - M  + j }{x} \right)^{k-1} \frac{1}{x} \chi_{[0, M-j ] } ( \ell) d \ell . 
\end{align*} 
As a consequence
\begin{align*}
    \nu ( d \ell) &= e^{- \frac{M}{x}} \delta(\ell - M )+  e^{- \frac{M}{x}} \sum_{k=1}^M \binom{M}{k}  \sum_{j=0}^k I_{k, j } (x,z; \ell )
\end{align*}
and the statement follows. 
\end{proof}

 \subsection{$M \gg  1$ as $N \to \infty$} \label{app:energy M>>1}
In this section we prove Theorem \ref{thm:gaussian ATP}. 
To this end, it is convenient to introduce the generating function  associated to the probability densities of $ \{ m_k (\infty) \}_{k=0}^{N}$, i.e.
\[ 
G(z):= \sum_{k=0}^N z^k m_k(\infty ), \quad z \in \mathbb C. 
\]

\begin{lemma} \label{lem:expression for G}
Let $\tau >0$.
The generating function $G$ of $ \{ m_k (\infty) \}_{k=0}^{N}$ is given by 
    \begin{equation} \label{G}
        G(z)=\frac{z_0-1 }{z_0-z } \left( 1 + \left( \frac{z}{z_0} \right)^N \tau (1-z)  \right), \quad z \in \mathbb C 
    \end{equation}
    where 
    \begin{equation} \label{z0}
    z_0:= 1+ \frac{1}{\tau}. 
    \end{equation}
\end{lemma}
\begin{proof}
The definition of $G$, implies that 
\begin{align} \label{computation G}
G(z)= \left( \frac{ z \tau }{1+\tau } \right)^N + \sum_{k=0}^{N-1} \left(\frac{1}{1+\tau } \right) \left( \frac{\tau z }{1+\tau }\right)^k = a^N + \frac{S}{1+\tau },   
\end{align} 
where 
\begin{equation} \label{a}
a:=\frac{ z \tau }{1+\tau } 
\end{equation} 
and $S:= \sum_{k=0}^{N-1} a^k $. 
Since $S= \frac{1-a^N }{1-a }. $
we have that \eqref{computation G} can be rewritten as 
\begin{align*}
G(z)= a^N +  \frac{1}{1+\tau } \left(\frac{1-a^N} {1-a}  \right) =  \frac{1}{(1+\tau) (1-a) } \left( 1+ a^N \left( \tau -\tau a -a \right) \right).  
\end{align*} 
Using the fact that $a $ is given by \eqref{a} we deduce that 
\begin{equation} \label{G almost final}
G(z)= \frac{ 1+ a^N \tau (1-z) }{1+\tau(1-z) }. 
\end{equation}
Equality \eqref{G} follows by substituting \eqref{z0} and \eqref{a} in \eqref{G almost final}. 
\end{proof}
\begin{lemma} \label{lem:asympt of G}
Assume that $\lim_{N \to \infty} M = \infty $, that $\tau $ is given by \eqref{tau M to infty} and that $|z-1 | \leq   \frac{\log(M)}{ N \sqrt{M} } $.
Then 
    \[
\lim_{N \to \infty} \left( 1 + \left( \frac{z}{z_0} \right)^N \tau (1-z) \right)^M =1. 
    \]
    where $z_0$ is given by \eqref{z0}. 
\end{lemma}
\begin{proof}
    First of all we prove that
    \[ 
    \lim_{N \to \infty} \left|M \left( \frac{z}{z_0} \right)^N \tau (1-z)\right| =0. 
    \]  
    Notice that by  the definition of $z_0$, \eqref{z0}, we have that $\frac{1}{z_0} = \frac{\tau}{1+\tau }$. Therefore, as in \eqref{exp from tau} in the proof of Theorem \ref{thm: M >>1}, using that $\tau$ satisfies \eqref{tau M to infty}, we deduce that 
    \[
    \left(\frac{\tau}{1+\tau } \right)^N \leq  e^x\frac{1}{M}. 
    \] 
    As a consequence 
    \begin{align*}
\left| M \left( \frac{z}{z_0} \right)^N \tau (1-z) \right| = \left|  M z^N \left(\frac{\tau}{1+\tau } \right)^N \tau (1-z) \right|  \leq e^x  \left| z^N \tau (1-z) \right| \leq   e^x \tau  e^{ \frac{\log(M)}{\sqrt{M}}}  \frac{\log(M) } { N \sqrt{M} } . 
    \end{align*}
Since $\tau $ is given by \eqref{tau M to infty}, the asymptotic behaviour of $\tau_c $ in Proposition \ref{lem:asympt tau_c M less than expo} implies that
    \[
 \lim_{N \to \infty} \left| M \left( \frac{z}{z_0} \right)^N \tau (1-z) \right| =0.  
\] 
As a consequence 
    \begin{align*}
\lim_{N \to \infty} \left( 1 + \left( \frac{z}{z_0} \right)^N \tau (1-z) \right)^M &= \lim_{N \to \infty} \exp \left[ M \log \left( 1 + \left(  \frac{z}{z_0}\right)^N \tau (1-z)  \right)  \right] \\
&= \lim_{N \to \infty } \exp \left[  M \left(  \frac{z}{z_0}\right)^N \tau (1-z)   \right] = 1 . 
    \end{align*} 
\end{proof}

\begin{proof}[Proof of Theorem \ref{thm:gaussian ATP}]
Since $G(z)^M $ is the generating function of $Q_k $ to prove the theorem we want to compute
    \begin{align*}
    Q_k &= \frac{1}{2 \pi i } \oint_{\partial B_\delta(0)} \frac{G(z)^M } {z^{k+1} } dz \text{ as } N \rightarrow \infty \text{ for } k = \frac{\tau M }{\sqrt{(1+\tau) \tau M }}, 
    \end{align*}
    where $\delta < z_0$. 
    First of all we notice that the main contribution to this integral is due to the region of values of $z$ around the value $\hat{z}$ at which $ \left| \frac{G(z)^M }{z^k }\right| $ reaches the maximum.
    Notice that when $ k \sim \tau M $ as $N \to \infty $ we have that $\hat{z} \rightarrow 1$, hence $ \left| \frac{G(z)^M }{z^{1+k} }\right| \leq \frac{1}{z}$ for every $z \in \mathbb C$. Therefore deforming the contour as $\Re(z) \rightarrow - \infty $ and Lemma \ref{lem:asympt of G} we deduce that, if $\delta \leq  1+ \frac{\log(M)}{N \sqrt{M}} $, then 
    \begin{align*}
    Q_k &= \frac{1}{2 \pi i } \oint_{\partial B_\delta(0)} \frac{G(z)^M } {z^{k+1} } dz \sim 
     \frac{1}{2 \pi i } \oint_{\partial B_\delta(0)} \frac{ \left(  G_\infty (z)\right)^M } {z^{k+1} }  dz \text{ as } N \rightarrow \infty, 
    \end{align*}
    where $G_\infty (z)= \frac{z_0-1}{z_0-z}$. 
    By Cauchy Residue Theorem we deduce that
     \begin{align*}
 \frac{1}{2 \pi i } \oint_{\partial B_\delta(0)} \frac{\left(  G_\infty (z)\right)^M } {z^{k+1} }   dz &= Res \left( \frac{\left(  G_\infty (z)\right)^M } {z^{k+1} },  0  \right)
 = \frac{1}{k !} {\frac{d^{k} }{d z^k} (G_\infty (z))^M }_{| z=0}
  \\
  &= \frac{M (M+1) \dots (M+k-1)}{k !} (z_0-1)^M \left( \frac{1}{z_0}\right)^{M+k} \\
  &= \binom{ M+k -1 }{ k  } \left(1 -\frac{1}{z_0} \right)^M \left( \frac{1}{z_0}\right)^{k}. 
    \end{align*}  
    Since $\lim_{N \to \infty} M = \infty $ as well as $\lim_{N \to \infty} k = \infty $ we apply Stirling formula to deduce that 
    \begin{align*}
\binom{ M+k -1 }{ k  } \left(1 -\frac{1}{z_0} \right)^M \left( \frac{1}{z_0}\right)^{k} & = \frac{M }{ M + k }  \binom{ M+k }{ k  } \left(1 -\frac{1}{z_0} \right)^M \left( \frac{1}{z_0}\right)^{k} \\
& \sim \frac{M }{ M + k }  \sqrt{\frac{M+k }{ 2 \pi k M}} \left( \frac{k + M }{ k z_0 } \right)^k \left(  \frac{k + M }{ M } \left( 1- \frac{1}{z_0}\right) \right)^M \text{ as } N \rightarrow \infty. 
    \end{align*}
    We now perform the change of variables $\ell = \frac{k - \tau M }{\sqrt{M\tau (1+\tau) }}$ and use the fact that $z_0$ is given by \eqref{z0} to deduce that 
    \[
  \left(  \frac{k + M }{ M } \left( 1- \frac{1}{z_0}\right) \right)^M   = \exp \left( M \ln \left(  1+ \ell \sqrt{ \frac{\tau}{M(1+\tau )}}  \right)  \right) \sim \exp \left(  \ell \sqrt{  \frac{M\tau}{1+\tau }}  \right), \text{ as } N \to \infty. 
    \]
     On the other hand, 
     \begin{align*}
  \left(  \frac{k + M }{ k }  \frac{1}{z_0} \right)^k  &= \exp \left( - (M \tau + \ell\sqrt{M\tau (1+\tau)}) \ln \left( \frac{k z_0} {M+k } \right) \right) \\
  &\sim  \exp \left( - (M \tau + \ell\sqrt{M\tau (1+\tau)})\ln \left( \frac{k z_0} {M+k } \right) \right) \\
  & \sim  \exp \left( - (M \tau + \ell\sqrt{M\tau (1+\tau)})\frac{ \ell \sqrt{M (1+\tau) \tau }  } { M \tau (1+\tau)} \right) \\
  & \sim  \exp \left( -\left( \ell^2 + \ell \sqrt{\frac{ M\tau} { 1+ \tau }} \right)  \right) \text{ as } N \rightarrow \infty. 
  \end{align*} 
    Combining the two asymptotics we deduce that for $\ell = \frac{k - \tau M }{\sqrt{M\tau (1+\tau) }}$ we have that 
      \begin{align*}
\binom{ M+k -1 }{ k  } \left(1 -\frac{1}{z_0} \right)^M \left( \frac{1}{z_0}\right)^{k} 
\sim \sqrt{\frac{ 1   }{ 2 \pi }} e^{ -  \ell^2  }\text{ as } N \rightarrow \infty. 
    \end{align*}
    Therefore the desire result follows. 
\end{proof}

  \section{Speed ($M \approx e^{b N }$ as $N \to \infty $)} \label{app:speed}
  The aim of this Section is to prove Theorem \ref{thm:time}. To this end we start by computing the Laplace transform of the function $\Psi_M$ given by \eqref{psi_M}. 

  \begin{lemma}
 Let $N \geq 1$, $M \geq 1$ and $\tau >0$. 
 Let $ \Psi_M$ be given by \eqref{psi_M}. 
 Then the Laplace transform of $\psi_M $ is given by 
 \begin{equation}\label{laplace psi_M}
\hat{ \Psi}_M  (z) = \frac{1}{1- \hat{s}(z)} \left(  \frac{1}{1+ z + \frac{1}{\tau} } \right)^N \frac{1- (\hat{s }( z ))^M }{1- (\hat{s }( 0 ))^M}, \quad z \in \mathbb C 
 \end{equation}
 where $\hat{s} $ is the Laplace transform of \eqref{eq:s} and is given by 
 \[
 \hat{s}(z)= \frac{1}{1+\tau  z } \left( 1- \left( \frac{1}{1+z + \frac{1}{\tau}}\right)^N \right),\  z \in \mathbb C. 
 \]
 \end{lemma}
 \begin{proof}
     First of all notice that \eqref{psi_M} implies that 
     \[
     \hat{\Psi}_M(z) = \frac{1}{p_M(\tau) } \hat{r}(z) \sum_{j =1}^M \hat{s}^{j-1}(z) 
     \]
     where $p_M(\tau ) = 1 - \left( 1- \left( \frac{\tau }{1+\tau }\right)^N \right)^M= 1- (\hat{s}(0))^M$ and
     \begin{equation} \label{hat s}
     \hat{s }(z)= \frac{1}{\tau} \sum_{k =1 }^{N} \hat{n}_k(z)  =  \frac{1}{\tau} \sum_{k =1 }^{N}  \left(\frac{1}{1+z+\frac{1}{\tau}} \right)^k=  \frac{1}{\tau}\left(\frac{1}{1+z+\frac{1}{\tau}} \right) \sum_{k =0 }^{N-1}  \left(\frac{1}{1+z+\frac{1}{\tau}} \right)^k=\frac{ \left( 1+z + \frac{1}{\tau} \right)^N -1 }{(1+ \tau z )\left( 1+ z + \frac{1}{\tau} \right)^N }
     \end{equation}
     and     
     \[ 
     \hat{r}(z)= \hat{n}_N (z)=\left(\frac{1}{1+z+\frac{1}{\tau}} \right)^N = 1- (1+ \tau z ) \hat{s } (z) 
     \] 
 Therefore 
 \begin{align*}
 \hat{\Psi}_M(z) &= \frac{1}{1-  (\hat{s} (0))^M } ( 1- (1+ \tau z ) \hat{s } (z)  ) \sum_{j =1}^M \hat{s}^{j-1}(z) = \frac{1- (\hat{s}(z))^M -  \tau z  \sum_{j =1}^M \hat{s}^{j}(z)}{1-  (\hat{s} (0))^M } \\
 &= \frac{1}{1-  (\hat{s} (0))^M }\left( 1- (\hat{s}(z))^M -  \tau z  \hat{s}(z) \frac{1- (\hat{s}(z))^M}{ 1 - \hat{s}(z)} \right) =  \frac{1- (\hat{s}(z))^M}{1-  (\hat{s} (0))^M }\left( 1 -   \frac{\tau z  \hat{s}(z)}{ 1 - \hat{s}(z)} \right) \\
 &= \frac{1- (\hat{s}(z))^M}{1-  (\hat{s} (0))^M }   \frac{ 1- (1+\tau z ) \hat{s}(z)}{ 1 - \hat{s}(z)}  =   \frac{1}{1- \hat{s}(z)} \left(  \frac{1}{1+ z + \frac{1}{\tau} } \right)^N \frac{1- (\hat{s }( z ))^M }{1- (\hat{s }( 0 ))^M}
 \end{align*}
 where in the last equality we have used \eqref{hat s}. 
 \end{proof}

 We now write the proof of Proposition \ref{lem:scaled time}. 
 This lemma will be later used in the proof of Theorem \ref{thm:time}. 
 \begin{proof}[Proof of Proposition \ref{lem:scaled time}]
     First of all we prove that 
     \begin{equation} \label{eq s(z) s(0)}
     \lim_{N \to \infty } e^{ \zeta } \left(  \frac{ \hat{s} \left(\frac{\zeta }{\tau M } \right) }{ \hat{s} (0)} \right)^M =1 
     \end{equation}
     Indeed 
     \begin{align*}
 \left(    \frac{\hat{s} \left(\frac{\zeta }{\tau M } \right) }{ \hat{s} (0)}  \right)^M &= \left( \frac{1}{1+ \frac{\zeta }{M }} \right)^M  \left(  \frac{ 1- \left( \frac{1}{1+\frac{\zeta}{\tau M } + \frac{1}{\tau}} \right)^N }{1- \left( \frac{1}{1+\frac{1}{\tau} } \right)^N } \right)^M \sim e^{- \zeta }  \left(  \frac{ 1- \left( \frac{1}{1+ \frac{1}{\tau}}\right)^N \left( \frac{1+ \frac{1}{\tau}}{1+\frac{\zeta}{\tau M } + \frac{1}{\tau}} \right)^N }{1- \left( \frac{1}{1+\frac{1}{\tau} } \right)^N } \right)^M \\
 & \sim  e^{- \zeta }  \left(  \frac{ 1- \left( \frac{1}{1+ \frac{1}{\tau}}\right)^N \exp \left( - \frac{N \zeta }{M (1+\tau) }\right)  }{1- \left( \frac{1}{1+\frac{1}{\tau} } \right)^N } \right)^M  \sim  e^{- \zeta }  \left(  \frac{ 1- \left( \frac{1}{1+ \frac{1}{\tau}}\right)^N\left( 1- \frac{N \zeta }{M (1+\tau) }\right)  }{1- \left( \frac{1}{1+\frac{1}{\tau} } \right)^N } \right)^M \\
 & \sim  e^{- \zeta }  \left(   1- \left( \frac{1}{1+ \frac{1}{\tau}}\right)^N+ \left( \frac{1}{1+ \frac{1}{\tau}}\right)^N  \frac{N \zeta }{M (1+\tau) }  \right)^M \\
 & \sim  e^{- \zeta }  \left(   1+ \left( \frac{1}{1+ \frac{1}{\tau}}\right)^N  \frac{N \zeta }{M (1+\tau) }  \right)^M \sim  e^{- \zeta } \exp \left[  \left( \frac{\tau }{1+\tau }\right)^N  \frac{N \zeta }{ (1+\tau) }  \right] \text{ as } N \to \infty. 
     \end{align*}
     Now notice that, as in the proof of Theorem \ref{thm: M >>1} we have that 
     \[
0 \leq  \left( \frac{\tau }{1+\tau }\right)^N  \frac{N \zeta }{ (1+\tau) }   \leq \frac{1}{M } \left( 1+ \frac{\xi }{N } \right)^N \frac{ N \zeta }{ 1+ \tau } \leq  e^\xi \frac{ N \zeta }{ M ( 1+ \tau )  } \rightarrow 0 \text{ as } N \to \infty 
     \] 
     Hence $
 \left(    \frac{\hat{s} \left(\frac{\zeta }{\tau M } \right) }{ \hat{s} (0)}  \right)^M \sim e^{- \zeta }
$ as $N \to \infty $. Therefore \eqref{eq s(z) s(0)} follows. 

Now recall that
\[
\left( \frac{\tau }{1+\tau }\right)^N \sim \frac{1}{M } e^{ \xi } \text{ as } N \to \infty. 
\]
Hence 
\begin{align*}
    1- \hat{s}\left( \frac{\zeta }{ \tau M }\right) = \frac{1}{M } \left( \frac{\zeta}{ 1+ \frac{\zeta }{M }}  + \frac{M }{ 1+ \frac{\zeta }{M }} \left( \frac{\tau }{1+\tau }\right)^N  \right) \sim \frac{1}{M } \left( \zeta+ e^{ \xi } \right) \text{ as } N \to \infty 
\end{align*}
Therefore using \eqref{laplace psi_M} we deduce that 
\begin{align*}
   \hat{ \Psi}_M \left( \frac{\zeta }{ \tau M } \right) &\sim \frac{M }{  \zeta+ e^{ \xi } } \left( \frac{1}{1+\frac{\zeta }{\tau M } + \frac{1 }{\tau }} \right)^N  \frac{1- (\hat{s}(0))^M e^{- \zeta } }{1- (\hat{s}(0))^M } \\
   & \sim 
 \frac{1}{ \zeta+ e^{ \xi } }  M\left( \frac{\tau}{1+\tau} \right)^N  \frac{1- (\hat{s}(0))^M e^{- \zeta } }{1- (\hat{s}(0))^M } \sim 
   \frac{e^{ \xi}}{ \zeta+ e^{ \xi }} \frac{1- (\hat{s}(0))^M e^{- \zeta } }{1- (\hat{s}(0))^M } \\
   & \sim    \frac{e^{ \xi}}{ \zeta+ e^{ \xi }} \frac{1- e^{- e^{ \xi}} e^{- \zeta } }{1-  e^{- e^{ \xi}} } \text{ as } N \to \infty
\end{align*}
where in the last step we used the fact that 
\[
\hat{s }(0)^M = \left( 1- \left(\frac{\tau }{1+\tau }\right)^N \right)^M \sim \exp \left( - M \left( \frac{\tau }{1+\tau }\right)^N \right) \sim \exp \left( - \exp ( \xi ) \right) \text{ as } N \to \infty. 
\]
 \end{proof}

\begin{proof}[Proof of Theorem \ref{thm:time}]

By Proposition \ref{lem:scaled time} we know that 
taking the limits as $| \xi | \rightarrow \infty $ in \eqref{asympt psi_M} we obtain that
\[
\lim_{\xi \to - \infty } \lim_{N \to \infty } \hat{\Psi}_M\left( \frac{ \zeta  }{ \tau M }\right)= \lim_{\xi \to - \infty } \frac{1- e^{- \zeta} e^{- e^{ \xi }}}{ 1 - e^{- e^{ \xi }} } \cdot \frac{e^{\xi }}{\zeta + e^{\xi }} = \frac{1- e^{- \zeta} }{\zeta }
\]
and 
\[
  \lim_{\xi \to  \infty } \lim_{N \to \infty }  \hat{\Psi}_M\left( \frac{ \zeta e^{\xi} }{ \tau M }\right) =  \lim_{\xi \to  \infty }  \frac{1- e^{- \zeta e^\xi } e^{- e^{ \xi }}}{ 1 - e^{- e^{ \xi }} } \cdot \frac{e^{ \xi }}{\zeta e^\xi  + e^{\xi }} = \frac{1}{1+\zeta }. 
\] 
Now notice that $ \widehat{ e^{-t} } (\zeta) =  \frac{1}{1+\zeta}$ and $\widehat{\chi_{[0, 1]} } (\zeta )= \frac{1- e^{- \zeta }}{\zeta } $. 
Since the convergence of the Laplace transforms of measures on compact sets implies the weak convergence of the measures (see \cite{feller1991introduction}) the desired conclusion follows.

\end{proof}

\textbf{Acknowledgements} The authors gratefully acknowledge the support by the Deutsche Forschungsgemeinschaft (DFG) through the collaborative research centre "The mathematics of emerging effects" (CRC 1060, Project-ID 211504053) and Germany's Excellence StrategyEXC2047/1-390685813.
The funders had no role in study design, analysis, decision to publish, or preparation of the manuscript.

\bigskip 

\bigskip

\textbf{E. Franco}: Institute for Applied Mathematics, University of Bonn,

\hspace{0.5 cm} Endenicher Allee 60, D-53115 Bonn, Germany

\hspace{0.5 cm} E-mail: franco@iam.uni-bonn.de

\hspace{0.5 cm} ORCID 0000-0002-5311-2124

\bigskip 

\textbf{J. J. L. Vel\'azquez}: Institute for Applied Mathematics, University of Bonn,

\hspace{0.5 cm} Endenicher Allee 60, D-53115 Bonn, Germany

\hspace{0.5 cm} E-mail: velazquez@iam.uni-bonn.de

\bibliographystyle{siam}

\bibliography{References}
\end{document}